\documentclass[a4paper]{amsart}
\usepackage{a4wide}
\usepackage{amsmath,amssymb,amsthm,xy}
\usepackage[UKenglish] {babel}

\def\a{\alpha} \def\b{\beta} \def\d{\delta}  \def\f{\varphi}
\def\l{\lambda} \def\s{\sigma}  
 
 \def\I{\mathbb{I}}
 \def\minus{\smallsetminus}
\def\({\left(} \def\){\right)} 
\def\<{\langle} \def\>{\rangle}

\def\inv{^{-1}}
\def\N{\mathbb{N}}

\renewcommand\ge{\geqslant}
\renewcommand\le{\leqslant}
\newcommand\ie{i.e.}

\newcommand\forget[1]{}

\newcommand\rned[1]{\left\lfloor#1\right\rfloor}

\newcommand{\tp}{\otimes}

\DeclareMathOperator{\spann}{span}

\DeclareMathOperator{\Aut}{Aut}

\newcommand{\chr}{\mathop{\rm char}\nolimits}
\newcommand{\nperp}{\not\perp}

\newenvironment{smatrix}{\left(\begin{smallmatrix}}{\end{smallmatrix}\right)}

\newcommand{\B}{\mathcal{B}}
\DeclareMathOperator{\rad}{\operatorname{rad}}
\DeclareMathOperator{\soc}{\operatorname{soc}}
\DeclareMathOperator{\modu}{\operatorname{mod}}

\newtheorem{thm}{Theorem}
\newtheorem{lma}[thm]{Lemma}
\newtheorem{prop}[thm]{Proposition}
\newtheorem{cor}[thm]{Corollary}

\theoremstyle{definition}

\theoremstyle{remark}
\newtheorem{rmk}[thm]{Remark}
\theoremstyle{definition}
\newtheorem{ex}[thm]{Example}

\xyoption{all}

\title{The Loewy length of a tensor product of modules of a dihedral two-group}
\author{Erik Darp\"o}
\address{Darp\"o: Graduate School of Mathematics, Nagoya University, Furocho, Chikusaku, Nagoya,
  Japan}
  \email{darpo@math.nagoya-u.ac.jp}
\author{Christopher C.~Gill}
\address{Gill: Department of Algebra, Charles University, Sokolovska~83, Praha~8, 186~75, Czech
Republic}
\email{gill@karlin.mff.cuni.cz}
\thanks{The first author acknowledges support from the Swedish Research Council, Grant
  no~623-2009-709.}

\begin{document}
\selectlanguage{UKenglish}
\date{}

\begin{abstract}
While the finite-dimensional modules of the dihedral $2$-groups over fields of
characteristic $2$ were classified over 30 years ago, very little is known about the
tensor products of such modules. 
In this article, we compute the Loewy length of the tensor product of two modules of a
dihedral two-group in characteristic $2$. 
As an immediate consequence, we determine when such a tensor product has a projective
direct summand.
\end{abstract}


\maketitle 

\section{Introduction}

The tensor product is an invaluable and frequently used tool in the representation theory
of finite groups.  Given a field $K$ and a finite group $G$, the co-algebra structure of
$KG$, defined by $\Delta(g)=g\tp g$ for all $g\in G$, gives rise to a tensor product on
the category $\modu KG$ of finite-dimensional $KG$-modules: 
$x\cdot(m\tp n)=\Delta(x)(m\tp n)$ for $m\in M$, $n\in N$ and $x\in KG$.
The tensor product of two indecomposable $KG$-modules is usually not indecomposable, and
the problem of determining a direct sum decomposition -- the Clebsch--Gordan problem -- is
extremely difficult and in general not well understood. 

One approach to studying the tensor product of $KG$-modules goes via the representation
ring, or Green ring, $A(KG)$, which encodes the behaviour of the tensor product in its
multiplicative structure. This approach was pioneered by J.~A. Green in \cite{green61},
who proved that the Green ring of a cyclic $p$-group is semi-simple. 
Much of the research on Green rings since has focussed on the question of semisimplicity,
asking for which group algebras $KG$ the Green ring $A(KG)$ contains nilpotent elements. 
Notably, Benson and Carlson \cite{Benson&Carlson}
provided a general method to construct nilpotent elements in Green rings, and defined an
ideal $A(KG;p)$ of $A(KG)$ (here $p=\chr K$) such that the quotient $A(KG)/A(KG;p)$ has no
nilpotent elements.
 
The most complete results concerning direct sum decompositions of tensor products are
for cyclic $p$-groups and the Klein four-group $V_4$.  The indecomposable modules of $V_4$
over a field of characteristic $2$ were first determined by Kronecker, and 
Conlon \cite{conlon65} computed the direct sum decompositions of tensor products of
such modules. Both results are surveyed in \cite{archer08}.
The indecomposable modules of cyclic $p$-groups over a field of characteristic $p$
correspond to Jordan blocks with eigenvalue $1$, with the tensor product of modules given
by the Kronecker product of matrices. 
The problem of decomposing tensor products of cyclic $p$-groups has been studied by several
authors \cite{green61,srinivasan64,iwamatsu07a,norman08,barry11}.
However, all solutions to this problem that have been published so far, to our knowledge,
are recursive; no closed formula for the decomposition seems to be known.

Let $k$ be a field of characteristic 2, and $D_{4q}$ the dihedral
group of order $4q$, where $q\ge2$ is a 2-power. The indecomposable $kD_{4q}$-modules were classified over thirty years ago by
Ringel \cite{ringel75}.  However, in contrast with the cyclic $p$-groups and $V_4$, the
behaviour of the tensor product of $kD_{4q}$-modules is not well understood.  The Clebsch--Gordan problem for $kD_{4q}$ remains far from being solved, and progress has been limited to
some special cases.
One example is the work \cite{bessenrodt91} by Bessenrodt, classifying all
\emph{endotrivial} $kD_{4q}$-modules, that is, modules $M$ with the property that 
$M^*\tp M$ is a direct sum of the trival module and a projective module.
Archer \cite{archer08} studied the Benson--Carlson quotient $A(kD_{4q})/A(kD_{4q};2)$
when $k$ is algebraically closed,
showing how multiplication in this quotient is related to the Auslander--Reiten quiver of
$kD_{4q}$, and realising the quotient as the integral group ring of an infinitely
generated, torsion-free abelian group.

In this article, using the classification of indecomposable modules, we determine the Loewy
length of the tensor product of any two finite-dimensional $kD_{4q}$-modules. This
provides an additional piece of information towards the understanding of the Green rings of
the dihedral $2$-groups, and gives certain bounds on which modules can occur as direct
summands of such a tensor product. 
As an application, we determine precisely which tensor products have maximal Loewy length, that is, which tensor products have projective direct summands.

Write $D_{4q}=\<\s,\tau\mid \s^2=\tau^2=(\s\tau)^{2q}=1\>$.  
Then 
\begin{equation} \label{isomorphism}
kD_{4q} \:\tilde{\to}\: \frac{k\<X,Y\>}{\left(\, X^2,\,Y^2,\,(XY)^q+(YX)^q \,\right)} 
\quad\mbox{via}\quad\begin{cases}\s\mapsto 1+X ,\\ \tau\mapsto 1+Y . \end{cases}
\end{equation}
In particular, every $kD_{4q}$-module is also a module of the algebra
$\Lambda_0=k\<X,Y\>/(X^2,\,Y^2)$ and conversely, every finite-dimensional
$\Lambda_0$-module is a module of $kD_{4q}$ for sufficiently large $q$.
From here on, all modules are assumed to be finite dimensional.
The algebras $kD_{4q}$ are special biserial, hence the indecomposable modules are of three
types: strings, bands and projectives.
Below we recollect the classification of the indecomposable $kD_{4q}$-modules, due to
Ringel \cite{ringel75}.

Let $\mathcal{W}$ be the set of words $a_1\cdots a_n$ ($n\ge0$), in the alphabet
$X,X\inv,Y,Y\inv$ with the property that if $a_i\in\{X,X\inv\}$ then
$a_{i+1}\in\{Y,Y\inv\}$ and if $a_i\in\{Y,Y\inv\}$ then $a_{i+1}\in\{X,X\inv\}$. The empty
word is denoted by $1$.
For any word $w=a_1\cdots a_n\in\mathcal{W}$, set $w\inv=a_n\inv\cdots a_1\inv$.
Take $\sim_1$ to be the equivalence relation on $\mathcal{W}$ identifying every word $w$
with its inverse $w\inv$.

Let $\mathcal{W}'\subset\mathcal{W}$ be the set of words $w$ with the following
properties:
\begin{enumerate}
\item $w$ has even, positive length,
\item $w$ is not a power of a word of smaller length,
\item $w$ contains letters from both  $\{X,Y\}$ and $\{X\inv,Y\inv\}$.
\end{enumerate}
Define an equivalence relation $\sim_2$ on $\mathcal{W}'$  by saying that
$w\sim_2 w'$ if, and only if, either $w$ or $w\inv$ is a cyclic permutation of $w'$.

Given a word $w=a_1\cdots a_m\in\mathcal{W}$, a $\Lambda_0$-module $M(w)$ is defined as
follows:
$M(w)=\bigoplus_{i=0}^m ke_i$, and the 
action of $Z\in\{X,Y\}$ on $M(w)$ is defined by
\begin{equation} \label{stringaction}
  Z\cdot e_i=
  \begin{cases}
    e_{i-1} &\mbox{if}\quad i>0,\: a_i=Z, \\
    e_{i+1} &\mbox{if}\quad i<m,\: a_{i+1}=Z\inv, \\
    0      &\mbox{otherwise}.
  \end{cases}
\end{equation}

It is often helpful to picture the module $M(w)$, where $w=l_1l_2\cdots l_m$, by a schema  
$$\xymatrix{ke_0&ke_1\ar^{l_1}[l]&ke_2\ar^{l_2}[l]&\cdots\ar^{l_3}[l] &ke_{m-1}\ar^{l_{m-1}}[l]&ke_m\ar^{l_m}[l]}$$
The practice is to change the direction of the arrows in the schema that represent inverted
elements: $l_i=X\inv$ or $l_i=Y\inv$.
 For example, if $w=XYXY^{-1}X^{-1}YX^{-1}Y^{-1}$, then the schema of $M(w)$ is written
 as
\[\xymatrix@!0{&&&{ke_3}\ar^Y[dr]\ar_X[dl]&&&&\\
&&{ke_2}\ar_Y[dl]&&{ke_4}\ar^X[dr]&&ke_6\ar_Y[dl]\ar^X[dr]&&\\
&ke_1\ar_X[dl]&&&&ke_5&&ke_7\ar^Y[dr]&\\
ke_0&&&&&&&&ke_8
}\]

Next, let $\f$ be an indecomposable linear automorphism of $k^n$, and 
$w=a_1\cdots a_m\in\mathcal{W}'$.
Define $M(w,\f)=\bigoplus_{i\in\{0,\ldots,m-1\}} V_i$, where 
$V_i=\bigoplus_{j\in\{0,\ldots,n-1\}} k e_j^{(i)} \simeq k^n$
for all $i$.
Denote by $T_\f: M(w,\f)\to M(w,\f)$ the linear automorphism given by 
$$T_\f\left(e_j^{(i)}\right)=
\begin{cases}
  e_j^{(i-1)} &\mbox{if}\quad i\ge2,\\
  \f\left(e_j^{(i-1)}\right) &\mbox{if}\quad i=1,\\
  e_j^{(m)}   &\mbox{if}\quad i=0,
\end{cases}
$$
where, for each $i$, the map $\f$ is viewed as a map $V_i\to V_i$ in the natural
way. 
Now set
\begin{equation} \label{bandaction}
  Z\cdot e_j^{(i)}=
  \begin{cases}
    T_\f\left(e_j^{(i)}\right)     &\mbox{if}\quad a_i=Z,\\
    T_\f\inv \left(e_j^{(i)}\right) &\mbox{if}\quad a_{i+1}=Z\inv,\\
    0                             &\mbox{otherwise}.
  \end{cases}
\end{equation}
This defines a $\Lambda_0$-modules structure on $M(w,\f)$.
Such a module can be illustrated by a schema in the following way:
\[\xymatrix{V_0\ar_{l_m}@/_1.5pc/[rrrrr]&V_1\ar_{l_1=\varphi}[l]&V_2\ar_{l_2}[l]&\cdots\ar[l] &V_{m-2}\ar_{l_{m-2}}[l]&V_{m-1}\ar_{l_{m-1}}[l]}\]
If $w=XYXY^{-1}X^{-1}YX^{-1}Y^{-1}$ then $M(w,\varphi)$ is given by the following schema:
\[\xymatrix@!0{&&&{V_3}\ar^Y[dr]\ar_X[dl]&&&&\\
&&{V_2}\ar_Y[dl]&&{V_4}\ar^X[dr]&&V_6\ar_Y[dl]\ar^X[dr]&&\\
&V_1\ar_{X=\varphi}[dl]&&&&V_5&&V_7\ar^{Y}@/^/[dlllllll]&\\
V_0&&&&&&&&
}\]

Modules isomorphic to $M(w)$, $w\in\mathcal{W}$ are called \emph{string modules}, while
the ones of the type $M(w,\f)$ for $w\in\mathcal{W}'$ and $\f$ an indecomposable linear automorphism of $k^n$, are called
\emph{band modules}.  

Now, by Ringel's result, every indecomposable $\Lambda_0$-module is isomorphic to either a string
module or a band module. The two classes are mutually disjoint.
Moreover, $M(w)\simeq M(w')$ if, and only if, $w\sim_1 w'$, and $M(w,\f)\simeq M(w',\f')$
if, and only if, $w\sim_2 w'$ and $\f'=\psi\f\psi\inv$ for some $\psi\in\Aut_k(k^n)$. 

A $\Lambda_0$-module $M$ is in $\modu(kD_{4q})$ if, and only if,
$\left((XY)^q+(YX)^q\right)\cdot M=0$. This is clearly the case whenever the Loewy length
of $M$ is strictly less than $2q+1$, moreover, the regular module 
$_{kD_{4q}}kD_{4q}\simeq M((XY)^q(X\inv Y\inv)^q,1_k)$ is the unique indecomposable
projective, and the unique $kD_{4q}$-module with Loewy length equal to $2q+1$.

Throughout this article, the following notation and terminology is used.
The least natural number is $0$.
Given a non-negative real number $x$, $\rned{x}$ denotes the integral part of $x$, \ie,
$\rned{x}=\max\{n\in\N \mid n\le x\}$. 
Let $n\in\N$. The $i$th term in the binary expansion of $n$ is denoted by
$[n]_i$, so $n=\sum_i [n]_i2^i$. Further,
$\nu(n)=\min\{i\in\N \mid [n]_j=0,\,\forall j<i\}= \max\{i\in\N \mid \: 2^i\!\mid\! n\}$.
Given $l,m\in\N$, we write $l\perp m$ to indicate that the binary expansions of $l$ and $m$
are disjoint, that is, $[l]_i+[m]_i\le1$ for all $i\in\N$. All congruences appearing are
modulo two, so $l\equiv m$ always means $2\mid (l-m)$. 
By $\d_{i,j}$ we denote the Kronecker delta.
By a \emph{directed subword} of a word $w\in\mathcal{W}$ we mean a word $w'$ in either
the letters $\{X,Y\}$ or $\{X\inv,Y\inv\}$ such that $w=w_1w'w_2$ for some words 
$w_1,w_2\in\mathcal{W}$. A \emph{directed component} of $w$ is a maximal directed
subword. Clearly, every word in $\mathcal{W}$ can written in a unique way as a product of its
directed components. Moreover, for every word $w\in\mathcal{W}'$ there exists a word
$w'\in\mathcal{W}'$ with an even number of directed components such that $w\sim_2 w'$, and
the directed components of $w'$ are uniquely determined by $w$.
Define words
\begin{equation*}
  A_0=B_0=1, \;\mbox{and} \quad A_{t+1}=B_tY, \quad B_{t+1}=A_tX 
\quad\mbox{for all $t\in\N$.}
\end{equation*}
Then, for every directed word $w\in\mathcal{W}$ of length $t$, either $w\sim_1A_t$ or
$w\sim_1 B_t$ holds.

A basis of a $kD_{4q}$-module $M$ on which the algebra acts according to either of the
formulae (\ref{stringaction}) and (\ref{bandaction}) is called a \emph{standard basis} of
$M$. If $\B_M$ and $\B_N$ are standard bases of modules $M$ respectively $N$, the basis
$\B_M\tp \B_N=\{a\tp b \mid a\in\B_M, b\in\B_N\}$ of $M\tp N$ is called a 
\emph{homogeneous basis}.
We say that an element $a\tp b$ in a homogeneous basis is \emph{pure} if it is annihilated
by either $X$ or $Y$, and \emph{impure} otherwise.
By a \emph{subquotient} of a module $M$ we mean a quotient of a submodule of $M$,
equivalently, a submodule of a quotient of $M$.
The \emph{top} of a module $M$ is the quotient module $M/\rad M$, the \emph{socle},
$\soc M\subset M$, is the maximal semisimple submodule.

The \emph{Loewy length} of a module $M$ is denoted by $\ell(M)$. It is the common length
of the radical series and the socle series of $M$, so it may be computed as 
$\ell(M)=\min\{n\in\N \mid \rad^n(kD_{4q})\cdot M=0\}$.  We repeatedly make use of the fact that if $M,N$ are $kD_{4q}$-modules, and $M'$ is a subquotient of $M$, then $M'\tp N$ is a subquotient of $M\tp N$ and hence $\ell(M'\tp N)\leq \ell(M\tp N).$
The Loewy length of a string module $M(w)$ is $\ell(M)=h+1$, where $h$ denotes the maximal
length of all directed subwords of $w$. 
For band modules, we have $\ell(M(w,\f))=h'+1$, where $h'$ is the maximal length of all
directed subwords of any cyclic permutation of $w$.
For computational purposes, the numbers $h$ respectively $h'$ are often easier to work
with than with the Loewy length, therefore we define $h(M)=\ell(M)-1$ for $M\in\modu(kD_{4q})$. 
If $M$ is a module and $x\in M$, we write $\ell(x)=\ell(\<m\>)$ and similarly 
$h(x)=h(\<x\>)$.

A result which is crucial for the computational parts of this article is Lucas' theorem \cite{Lucas} (see also Exercise~6(a) in Chapter~1 of Stanley's book \cite{Stanleyvol1}).
Stated below for the special case of $p=2$, it will be used throughout the text without
further reference.
\begin{thm}[Lucas' theorem]
For all natural numbers $r$ and $s$, the congruence relation 
$$\binom{r}{s} \equiv \prod_{i\in\N}\binom{[r]_i}{[s]_i}$$
holds.
\end{thm}

The layout of this article is as follows.
In Section~\ref{results}, we state the main results, giving explicit, closed formulae for the Loewy
length of the tensor product of any two modules in the classifying list.
The first result, Proposition~\ref{subquo}, reduces the problem of determining the Loewy
length of a tensor product $M\tp N$ to the case when $M$ and $N$ both have simple top and
simple socle.  This works for modules for arbitrary finite groups.  In Proposition \ref{subquocor} we refine the result in the case of dihedral $2$-groups, showing that if $M$ and $N$ are indecomposable, and $M$ does not have simple top and simple socle, then the Loewy length of $M\tp N$ is the maximum of the Loewy lengths of $M_i\tp N$ where the $M_i$ are uniserial string modules corresponding to the directed components of the word defining $M$.  This reduces the calculation of the Loewy length of any tensor product of modules for $D_{4q}$ to determining the Loewy lengths of tensor products of some explicitly defined modules with simple top and simple socle.  
Thereafter, formulae for the Loewy length of a tensor product of such modules are given in
Theorem~\ref{mainthm}.
The proof of Propositions~\ref{subquo} and \ref{subquocor} are relatively short, and given in a few steps
in Section~\ref{results}.
As for Theorem~\ref{mainthm}, its proof occupies the remaining part of the article.
Sections~\ref{quiversetup}--\ref{soluniser} treat tensor products of string modules, while 
in Section~\ref{non-uniserial}, the Loewy lengths of products involving band modules are
computed. 
The basic setup of the problem for string modules is given in Section~\ref{quiversetup}.
In Section~\ref{bdlma} we prove an important auxiliary result, Proposition~\ref{backdiag},
which paves the way for proof of Theorem~\ref{mainthm}:1 in Section~\ref{soluniser}.
Finally, the formulae involving bands with simple top and simple socle are proved in
Section~\ref{non-uniserial}, mainly using the result for string modules from
Section~\ref{soluniser}.

\section{Results} \label{results}

Let $\Lambda$ be an algebra over a field $K$, and $M$ a $\Lambda$-module. 
Denote by $\pi:M\to M/\rad M$ the canonical projection.
Since the top of any module is semi-simple, there exists a basis $\tilde{\B}$ of 
$M/\rad M$ such that $\<\bar{b}\>\subset M/\rad M$ is simple for each
$\bar{b}\in\tilde{\B}$. 
Choose a set $\B'\subseteq M$ of coset representatives for the elements in $\tilde{\B}$;
these are now a set of linearly independent in $M$, and $M=\<\B'\>=\sum_{b\in\B'}\<b\>$.
Moreover, each of the submodules $\<b\>$, $b\in\B'$, has simple top $\<\pi(b)\>$.
Now $\B'$ can be extended to a basis $\B$ of $M$ such that each $b\in\B$ is contained in
$\<b'\>$ for some $b'\in\B'$.
The basis $\B$ now has the following properties:
\begin{enumerate}
\item it contains a subset $\B'$ such that the elements $\pi(b)$, $b\in\B'$ form a basis
  of $M/\rad M$, each element of which generates a simple submodule;
\item each element of $\B$ is contained in $\<b'\>$ for some $b'\in\B'$. 
\end{enumerate}
A basis satisfying these properties shall be called a \emph{good basis} of $M$.
A subset $\B'$ of $M$ satisfying the first condition is called a \emph{top basis} of $M$.
The preceding construction shows that good bases always exist, and that every top basis can
be extended to a good basis. 

\begin{lma}\label{gbasesexist}
Every $\Lambda$-module has a good basis. 
\end{lma}

Observe that in a $kD_{4q}$-module, a standard basis is a good basis, while a homogeneous
basis of a tensor product $M\tp N$ in general is not. 

\begin{lma}\label{lem:Basisradical}
If $\Lambda$ is an Artin algebra, $M\in\modu(\Lambda)$ and $X\subset M$ a set of generators
of $M$, then 
\[\ell(M)=\max_{x\in X}\ell(x).\]
\end{lma}
\begin{proof}
Let $r=\max_{x\in X}\ell(x)$.
The inequality $\ell(M)\ge r$ is immediate.
On the other hand, any element $m\in M$ can be written as $m=\sum_{x\in X}m_x$ for some
$m_x\in\<x\>$, so $(\rad^r\Lambda) m\subset\sum_{x\in X}(\rad^r\Lambda) m_x=0$.
Hence $r\le \ell(M)$.
\end{proof}

Let $M$ and $N$ be modules, and let $\B_M\subseteq M$ and $\B_N\subseteq N$ be good bases,
with top bases $\B'_M\subseteq\B_M$ and $\B'_N\subseteq\B_N$ respectively.
By Lemma~\ref{lem:Basisradical}, 
$\ell(M\otimes N)=\max \{ \ell(\<a\otimes b\>) \mid (a,b)\in\B_M\times\B_N \}$.
However,
$\<a\otimes b\>\subseteq\<a\>\otimes \< b\>\subseteq \<a'\>\tp\<b'\>\subseteq M\otimes N$
for some $a'\in\B'_M$, $b'\in\B'_N$, so
\begin{align*}
\ell(M\otimes N)=&\max\{ \ell(\< a\otimes b\>) \mid (a,b)\in\B_M\times\B_N\} \le
\max\{ \ell(\<a\>\otimes\<b\>) \mid (a,b)\in\B_M\times\B_N \}  \\
\le&\max\{ \ell(\<a'\>\otimes\<b'\>) \mid (a',b')\in\B'_M\times\B'_N \} \le \ell(M\otimes N) \\
\intertext{\ie,}
 \ell(M\tp N) =& \max\{ \ell(\<a'\>\otimes\<b'\>) \mid (a',b')\in\B'_M\times\B'_N\}.
\end{align*}

\begin{prop}\label{subquo}
If $K$ is a field, $G$ a finite group and $M,N\in\modu KG$, then $\ell(M\otimes N)$ is the
maximum of $\ell(A\otimes B)$ where $A,B$ are subquotients of $M$ and $N$ respectively with
simple top and simple socle.
\end{prop}
\begin{proof}
By the preceding discussion and Lemma~\ref{gbasesexist}, we may assume $M$ and $N$ have
simple top. 
On the other hand, $\ell(M\otimes N)=\ell((M\otimes N)^*)=\ell(M^*\otimes N^*)$ and, again by
applying the argument preceding this lemma, $\ell(M^*\tp N^*)$ is the maximum of $\ell(U\tp V)$,
where 
$U\subset M^*$, $V\subset N^*$ run though all submodules with simple top. Since $M^*$,
$N^*$ have simple socle, so have $U$, $V$.
Now $\ell(U\tp V)=\ell(U^*\tp V^*)$, and $U^*$ and $V^*$ are subquotients of $M$ and $N$
respectively, with simple top and simple socle. 
\end{proof}

The $\Lambda_0$-modules with simple top and simple socle are precisely the uniserial
string modules, that is, those isomorphic to $M(A_t),M(B_t)$ for $t\in\N$, and the band
modules isomorphic to $M(A_lB_m\inv,\rho)$ for $l,m\in\N,\rho\in k\minus\{0\}$.

While Proposition~\ref{subquo} is valid for all finite group algebras, we can do slightly
better with $kD_{4q}$.

\begin{prop}  \label{subquocor}
Let $N\in\modu(kD_{4q})$. 
  \begin{enumerate}
  \item If $w\in\mathcal{W}$ is a word with directed components $w_i$,
    $i\in\{1,\ldots,m\}$, and $M=M(w)$, then 
    $$\ell(M\tp N)=\max\{\ell(M(w_i)\tp N) \mid i\in\{1,\ldots,m\} \}.$$
  \item Let $m$ and $n$ be positive integers, $w\in\mathcal{W}'$ is a word with directed
    components $w_i$, $i\in\{1,\ldots,
    2m\}$, $\f$ an indecomposable automorphism of $k^n$, and $M=M(w,\f)$.
    If $m$ and $n$ are not both equal to $1$, then
    $$\ell(M(w,\f)\tp N)=\max\{\ell(M(w_i)\tp N) \mid i\in\{1,\ldots,2m\}
    \}.$$
  \end{enumerate}
\end{prop}

Note that if $m=n=1$ in the second statement above, then $M$ itself has simple top and
simple socle. 

\begin{proof}
In both cases of the proposition, under the respective assumptions, the modules $M(w_i)$ are
subquotients of $M$. Hence $\ell(M\tp N)\ge \max\{M(w_i)\tp N\}_i$.
On the other hand, $M$ itself is a subquotient, in the first case of 
$\bigoplus_{i} M(w_i)$, and in the second case of 
$\left(\bigoplus_{i} M(w_i)\right)^{\oplus n}$. Either way, it follows that $\ell(M\tp
N)\le \max\{\ell(M(w_i)\tp N)\}_i$, proving the assertion.
\end{proof}

\begin{rmk} \label{projquo}
  Let $G$ be a $p$-group and $K$ a field of characteristic $p$. Then the regular module
  $_{KG}KG$ is indecomposable, and a $M\in \modu KG$ contains a projective direct summand
  precisely when $\ell(M)=\ell(_{KG}KG)$. From Proposition~\ref{subquo} follows that for
  $M,N\in\modu KG$, the tensor product $M\tp N$ has a projective direct summand if, and
  only if, there exist subquotients $A$ and $B$, of $M$ and $N$ respectively, with simple
  top and simple socle, such that $A\tp B$ has a projective direct summand.
  In the case of $KG=kD_{4q}$, Proposition~\ref{subquocor} specifies precisely
  which subquotients $A$ and $B$ need to be considered. 
\end{rmk}

Given $l,m\in\N$, let $s\in\N$ be the smallest number such that $[l]_r+[m]_r\le1$ for
all $r\ge s$, and $\l=\sum_{i\ge s}[l]_i$, $\mu=\sum_{i\ge s}[m]_i$.
Define a binary operation on $\N$ by $l\#m=\l+\mu+2^s-1$.
Observe that $l,m\le l\#m\le l+m$, with $l\#m=l+m$ if, and only if $s=0$, that is, $l\perp m$. 

\begin{ex}
The binary expansions of $146$ and $1304$ are $146=2+2^4+2^7$ and  $1304=2^3+2^4+2^8+2^{10}$.  Thus, $146\nperp 1304$, and in this case $s=5$. It follows that $146\#1304 = 2^{10}+2^8+2^7+(2^5-1) = 1439$.
\end{ex}

We are now ready to state the main theorem of this article, which gives the Loewy lengths
of tensor products of modules with simple top and simple socle. The remaining sections are
dedicated to the proof of this theorem. 

\begin{thm} \label{mainthm}
Let $l,m\in\N$, $l_1,l_2,m_1,m_2\in\N\minus\{0\}$, $\rho,\s\in k\minus\{0\}$. 
\begin{enumerate}
\item \emph{String with string:}
  \begin{align*}
    \ell(M(A_l)\tp M(B_m))&=
    \begin{cases}
      1+l\#m = 1+l+m &\mbox{if}\quad l\perp m, \\
      2+l\#m     &\mbox{if}\quad l\not\perp m,
    \end{cases} \\
    \ell(M(A_l)\tp M(A_m)) &=
    \begin{cases}
      1+l\#m 
      &\mbox{if $[l]_t=[m]_t=0$ for all $0\le t<s-1$},\\    
      2+l\#m &\mbox{otherwise}.
    \end{cases}\\
\intertext{where $s=\min\{r\in\N\mid [l]_t+[m]_t\le1,\:\forall t\ge r\}$.} 
  \end{align*}
\item \emph{Band with string:}
  $$  \ell\left(M(A_{l_1}B_{l_2}\inv,\rho)\tp M(A_m)\right) =
\begin{cases}
  2+(l_1-1)\#m & \mbox{if $\rho=1$, $l_1=l_2$ and} \\
               &\hspace{2ex}\mbox{$l_1\perp m$, $l_1\perp(m-1)$,} \\
  \ell\left(M\left(A_{l_1}B_{l_2}\inv \right)\tp M(A_m)\right) & \mbox{otherwise.}
\end{cases}
$$
\item \emph{Band with band:} 
Let $M=M\left(A_{l_1}B_{l_2}\inv,\rho\right)$, $N=M\left(A_{m_1}B_{m_2}\inv,\s\right)$.
\vspace{1ex}
  \begin{enumerate}
  \item 
    If $l_1\ne l_2$, then 
    $$\ell(M\tp N)=\ell(M\left(A_{l_1}B_{l_2}\inv\right)\tp N).$$
  \end{enumerate}
Assume $l_1=l_2$, $m_1=m_2$.
\vspace{1ex}
\begin{enumerate}\addtocounter{enumii}{1}
\item 
  If $l_1\nperp m_1$, $l_1\nperp(m_1-1)$, $(l_1-1)\nperp m_1$ then 
  $$\ell(M\tp N)= 2+(l_1-1)\#(m_1-1)=2+l_1\#m_1.$$
\item 
  If $l_1\perp m_1$, $(l_1-1)\perp m_1$, then
  $$\ell(M\tp N)=
  \begin{cases}
    2+(l_1-1)\#(m_1-1) &\mbox{if}\quad \s=1,\\
    l_1+m_1+1 &\mbox{otherwise.}
  \end{cases}$$
\item 
  If $l_1\perp m_1$, $l_1\perp(m_1-1)$, then
  $$\ell(M\tp N)=
  \begin{cases}
    2+(l_1-1)\#(m_1-1) &\mbox{if}\quad \rho=1,\\
    l_1+m_1+1 &\mbox{otherwise.}
  \end{cases}$$
\item 
  If $(l_1-1)\perp m_1$, $l_1\perp(m_1-1)$, then
  $$\ell(M\tp N)=
  \begin{cases}
    2+(l_1-1)\#(m_1-1) &\mbox{if}\quad \rho=\s=1,\\
    l_1+m_1            &\mbox{if}\quad \rho=\s\ne1,\\
    l_1+m_1+1 &\mbox{otherwise.}
  \end{cases}$$
  \end{enumerate}
\end{enumerate}
\end{thm}
We remark that if any of the statements $l\perp m$, $(l-1)\perp m$ and $l\perp(m-1)$ holds
true, then so does precisely one of the remaining two. Hence the cases listed in item 3
above are all possible.
Furthermore, in cases 1 and 2 of Theorem \ref{mainthm}, the identities obtained by
interchanging the letters $A$ and $B$ also hold true. This can be seen by observing that
$X\mapsto Y$, $Y\mapsto X$ defines an automorphism of $kD_{4q}$, sending $A_t$ to $B_t$
and vice versa. 

We can now answer the question of when the tensor product of two $kD_{4q}$-modules
contains a projective direct summand. By Proposition~\ref{subquocor}
it suffices to consider modules with simple top and simple
socle. Hence we need only to read off from Theorem~\ref{mainthm} when the tensor product
of two such modules has Loewy length $2q+1$.

\begin{cor} \label{projsumm}
Let $l,m<2q$, $0<l_1,l_2,m_1,m_2<2q$, and $\rho,\sigma\in k\smallsetminus\{0\}$.
\begin{enumerate}
\item $M(A_l)\otimes M(B_m)$ has a projective direct summand if, and only if, 
  $l+m\ge 2q$, 
\item $M(A_l)\otimes M(A_m)$ has a projective direct summand if, and only if, $l+m\ge 2q+1$.
\item 
  $M(A_{l_1}B_{l_2}^{-1},\rho)\otimes M(A_m)$ has a projective
  direct summand precisely when 
  $$\max\{l_1+m-1,l_2+m\}\ge 2q.$$
\item 
  If $l_1\ne l_2$ or $m_1\ne m_2$, then
  $M\left(A_{l_1}B_{l_2}^{-1},\rho\right)\otimes M\left(A_{m_1}B_{m_2}^{-1},\sigma\right)$
  has a projective direct summand if, and only if,  
  $$\max\{l_1+m_1-1,l_1+m_2,l_2+m_1,l_2+m_2-1\}\ge 2q\,.$$
\item If $l_1=l_2$, $m_1=m_2$ then
  $M\left(A_{l_1}B_{l_2}^{-1},\rho\right)\otimes M\left(A_{m_1}B_{m_2}^{-1},\sigma\right)$
  has projective direct summands if, and only if,
  \begin{enumerate}
  \item
    $l_1\nperp(m_1-1)$, and $l_1+m_1\ge 2q$, or
  \item
    $l_1\perp(m_1-1)$, $\rho\ne\sigma$ and $l_1+m_1=2q$.
  \end{enumerate}
\end{enumerate}
\end{cor}

For $l,m<2q$, the condition $l+m\ge 2q$ implies $l\nperp m$.
Thus, in particular, in 5(a) above, the condition $l_1\perp(m_1-1)$ is equivalent to
$(l_1-1)\perp m_1$, and similarly, in 5(b), $l_1\nperp(m_1-1)$ could be replaced by
$(l_1-1)\nperp m_1$.

\begin{proof}
A $kD_{4q}$-module has a projective summand if, and only if, its Loewy length equals $2q+1$.
Observe that if $M$ and $N$ are $kD_{4q}$-modules with $\ell(M)=l+1$ and $\ell(N)=m+1$,
then $\ell(M\tp N)\le l+m+1$.
In particular, if $l\perp m$, then $\ell(M\tp N)\le l+m+1\le 2q$, so $M\tp N$ contains no
projective summands.  From here on we assume $l,m<2q$ and $l\nperp m$.

From the definition it is clear that $l\#m\le 2q-1$. 
Moreover, $l\#m=2q-1$ if, and only if, $[l]_r+[m]_r=1$ for all
$r\in\{s,s+1,\ldots,\log_2(q)\}$ (here $s\in\N$ is as in the definition of $l\#m$),
which is equivalent to $l+m\ge2q$. 
Now, it follows from Theorem~\ref{mainthm}:1 that $\ell(M(A_l)\tp M(B_m))=2q+1$ precisely when
$l+m\ge2q$. As for $M(A_l)\tp M(A_m)$, its Loewy length is $2q+1$ if, and only if,
$l+m\ge2q$ and there exists a $t<s-1$ such that $[l]_t=1$ or $[m]_t=1$. This is equivalent
to $l+m\ge2q+1$.

For 3, note that, by Proposition~\ref{subquocor} and 1--2 above, $\max\{l_1+m-1,l_2+m\}\ge 2q$ if, and only if, 
$M(A_{l_1}B_{l_2}\inv)\tp M(A_m)$ has a projective direct summand. 
Now Theorem~\ref{mainthm}:2 tells us that 
$\ell(M(A_{l_1}B_{l_2}\inv,\rho)\tp M(A_m))=\ell(M(A_{l_1}B_{l_2}\inv)\tp M(A_m))$,
unless $l_1=l_2$, $l_1\perp m$, $l_1\perp(m-1)$ and $\rho=1$. In the latter case,
$\ell(M(A_{l_1}B_{l_2}\inv)\tp M(A_m))\le l_1+m=\max\{l_1+m-1,l_2+m\}<2q$ and no
projective summands appear. This proves the result in this case.

Let $M=M(A_{l_1}B_{l_2}\inv,\rho)$ and $N=M(A_{m_1}B_{m_2}\inv,\s)$.
If $l_1\ne l_2$ then $\ell(M\tp N)=\ell(M(A_{l_1}B_{l_2}\inv)\tp N)$ by
Theorem~\ref{mainthm}:3(a), whence the result follows from 1--3 above.
The case $m_1\ne m_2$, of course, is analogous.

It remains to prove 5. Assume $l_1=l_2$ and $m_1=m_2$. Clearly, if $l_1\perp m_1$ then
$l_1+m_1<2q$ and $M\tp N$ has no projective summand. Suppose instead that $l_1\nperp m_1$.
If $l_1\nperp(m_1-1)$ then also $(l_1-1)\nperp m_1$, and Theorem~\ref{mainthm}:3(b) means
that $M\tp N$ has a projective summand if, and only if, $2+l_1\#m_1=2q+1$, that is, if and
only if $l_1+m_1\ge 2q$.
If $l_1\perp(m_1-1)$ then also $(l_1-1)\perp m_1$, and we are in the situation of
Theorem~\ref{mainthm}:3(e). The condition $(l_1-1)\perp m_1$ means that 
$2+(l_1-1)\#(m_1-1)\le l_1+m_1\le 2q$, consequently, $\ell(M\tp N)<2q+1$ whenever
$\rho=\s$.
If $\rho\ne\s$ then $\ell(M\tp N)=l_1+m_1+1$, so $M\tp N$ has a projective direct summand
precisely when $l_1+m_1=2q$.
\end{proof}

\begin{ex}
In many cases the Loewy length of a tensor products $M\otimes N$ of indecomposable modules $M$ and $N$ can be reduced to determining the maximum of the Loewy lengths of the tensor products of uniserial modules corresponding to the directed components of the defining words for $M$ and $N$, by Proposition \ref{subquocor} and Theorem \ref{mainthm}.  An example will be illuminating. We consider a tensor product of band modules $M=M(A_{l_1}B_{l_2}^{-1},\varphi)$ and $N=M(A_{m_1}B_{m_2}^{-1},\psi)$, where $\varphi$ and $\psi$ are indecomposable automorphisms of $k^{a}$ and $k^b$ respectively.
\begin{enumerate}
\item If $a,b>1$ then, by Proposition~\ref{subquocor},
\begin{multline*}
\ell(M\tp N)=\max\{\ell(M(A_{l_1})\tp M(A_{m_1})), \ell(M(B_{l_2})\tp M(A_{m_1})), \\
\ell(M(A_{l_1})\tp M(B_{m_2})),\ell(M(B_{l_2})\tp M(B_{m_2}))\} \,.
\end{multline*}

\item If $a>1$, $b=1$, then 
$\ell(M\tp N) = \max\{\ell(M(A_{l_1})\tp N),\ell(M(B_{l_2})\tp N)\}$,
 the value of which is given by Theorem~\ref{mainthm}:2.
Analogously, $\ell(M\tp N) = \max\{\ell(M\tp M(A_{m_1})),\ell(M\tp M(B_{m_2}))\}$ if
$a=1$, $b>1$.
\end{enumerate}
Assume that $a=b=1$, so $\varphi,\psi\in k\minus \{0\}$.
\begin{enumerate} \addtocounter{enumi}{2}
\item 
If $l_1\neq l_2$ then from Theorem \ref{mainthm}:3(a) and Proposition \ref{subquocor}
follows 
$$\ell(M\otimes N)=\max\{\ell(M(A_{l_1})\otimes N),\ell(M(B_{l_2})\tp N)\},$$
which is computed in Theorem~\ref{mainthm}:2.
In particular, if $\psi\neq 1$, a further application of
Proposition \ref{subquocor} reduces this once again to 
\begin{multline*}
\ell(M\tp N)=\max\{\ell(M(A_{l_1})\tp M(A_{m_1})), \ell(M(B_{l_2})\tp M(A_{m_1})), \\
\ell(M(A_{l_1})\tp M(B_{m_2})),\ell(M(B_{l_2})\tp M(B_{m_2}))\} \,.
\end{multline*}
Of course, the case $m_1\ne m_2$ is analogous.

\item If $l_1=l_2$ and $m_1=m_2$, then $\ell(M\tp N)$ is given directly by
  Theorem~\ref{mainthm}:3(b)--(e).
\end{enumerate}

We consider the case $l=146$ and $m=266$, and calculate
$\ell(M(A_lB_l\inv,1)\tp M(A_mB_m\inv,1))$.  The binary expansions of $l$ and $m$ are 
$$146 = 2+2^4+2^7 \quad \mbox{and} \quad 266= 2+2^3+2^8.$$
Thus, we have $(l-1)\perp m$ and $l\perp (m-1)$ whilst $l\nperp m$. 
Theorem \ref{mainthm}:3(e) now gives $\ell(M\otimes N)=2+(l-1)\#(m-1)=411.$  
In contrast, $l\#m=411$, and hence by Theorem \ref{mainthm}:1, we have 
\[\ell(M(A_{146})\tp M(A_{266}))= 412\mbox{ and } \ell(M(A_{146})\tp M(B_{266}))= 413.\]
\end{ex}

In fact, the difference between $\ell(M\otimes N)$ and the maximum of Loewy lengths of
tensor products of uniserial modules given by the directed components of the words
defining $M$ and $N$ can be arbitrarily large, as demonstrated in the following example. 

\begin{ex}
Let $l,m\in\mathbb{N}$ be such that $(l-1)\perp m$ and $l\perp(m-1)$, that is, $l=\lambda+2^a$,
$m=\mu+2^a$ with $a\in\mathbb{N}$, $\lambda\perp\mu$ and $2^{a+1}\mid \lambda,\mu$.
In particular, $\nu(l)=\nu(m)=a$.
Then, by Theorem~8:3(e),
\begin{align*}
&\ell(M(A_lB_l\inv,1)\tp M(A_mB_m\inv,1) ) = 2+ (l-1)\#(m-1) 
= 2 + \lambda +\mu + (2^a-1) , \\
\intertext{whilst}
&\max\{\ell(M(A_l)\tp M(A_m)), \ell(M(B_l)\tp M(A_m)),
\ell(M(A_l)\tp M(B_m)),\ell(M(B_l)\tp M(B_m))\} \\
&\hspace{2ex} = \ell(M(A_l)\tp M(B_m)) = 2+ l\#m = 2+ \lambda +\mu + (2^{a+1}-1) \\
&\hspace{2ex} = \ell(M(A_lB_l\inv,1)\tp M(A_mB_m\inv,1) ) + 2^a \,.
\end{align*}
\end{ex}

\section{Basic setup} \label{quiversetup}
In view of the isomorphism (\ref{isomorphism}), we may consider any $M\in\modu kD_{4q}$ as
a module of the algebra 
$k\<X,Y\>/\left(\, X^2,\,Y^2,\,(XY)^q+(YX)^q \,\right)$. The module structure of the
tensor product $M\tp N$ of two modules $M,N\in\modu kD_{4q}$ is given by 
\begin{align*}
X(m\tp n)=& Xm\tp n + m\tp Xn + Xm\tp Xn ,\\
%\intertext{and}
Y(m\tp n)=&Ym\tp n + m\tp Yn + Ym\tp Yn\end{align*}
for $m\in M$ and $n\in N$.  We analyse this action in terms of a quiver
representation, as is described in the remainder of this section. 
Define a quiver $\Gamma$ as follows:
\begin{align*}
 \Gamma_0&=\N\times\N \,, \qquad 
 \Gamma_1=\{\a_{i,j},\b_{i,j},\gamma_{l,m}\mid (i,j),(l,m)\in\N\times\N,\;l+m\in2\N\}\,, \\
 &\begin{cases}
   \a_{i,j}:(i,j)\to(i,j+1), \\
   \b_{i,j}:(i,j)\to(i+1,j), \\
   \gamma_{i,j}:(i,j)\to(i+1,j+1).
 \end{cases}
\end{align*}
Let $V$ be the characteristic representation of $\Gamma$, that is, the representation
obtained by setting $V_{(i,j)}=k$ for all vertices $(i,j)\in\Gamma_0$, and $V_a=\I_k$ for
all arrows $a\in \Gamma_1$.
We write $1_{(i,j)}$, or simply $(i,j)$ when the context is clear, for the identity
element of $k=V_{(i,j)}$ at the vertex $(i,j)\in\Gamma_0$. 
For $r,s\in\N$, let $V(r,s)=V/I(r,s)$ where 
$I(r,s)=\spann_k\{1_{(i,j)}\mid i>r \;\mbox{or}\; j>s\}\subset V$ is
the subrepresentation generated by all elements $1_{(r+1,j)}$ and $1_{(i,s+1)}$, $i,j\in\N$.

Set
\begin{align*}
X\cdot {(i,j)}=
\begin{cases}
  0                             & \mbox{if}\quad  i,j\equiv0, \\
  (i+1,j)                       & \mbox{if}\quad  i\equiv1,\:j\equiv0, \\
  (i,j+1)                       & \mbox{if}\quad  i\equiv0,\:j\equiv1, \\
  (i+1,j) + (i,j+1) + (i+1,j+1) & \mbox{if}\quad  i,j\equiv1,
\end{cases} \\
\intertext{and}
Y\cdot {(i,j)}=
\begin{cases}
  0                             & \mbox{if}\quad  i,j\equiv1, \\
  (i+1,j)                       & \mbox{if}\quad  i\equiv0,\:j\equiv1, \\
  (i,j+1)                       & \mbox{if}\quad  i\equiv1,\:j\equiv0, \\
  (i+1,j) + (i,j+1) + (i+1,j+1) & \mbox{if}\quad  i,j\equiv0.
\end{cases}
\end{align*}
This gives $V$ the structure of an (infinite-dimensional) $\Lambda_0$-module. Since
$I(r,s)$ is closed under the $\Lambda_0$-action, $V(r,s)$ also carries a
$\Lambda_0$-module structure induced from $V$.

Let $M=M(A_m)$ and $N=M(A_n)$. 
The two modules have standard bases $\{A_r(u)\}_{r\in\{0,\ldots,m\}}$ and
$\{A_s(v)\}_{s\in\{0,\ldots,n\}}$ respectively, where $u\in M$ and $v\in N$ are top basis
elements.
Define linear maps $\f:M\tp N \to V(m,n)$ by $\f(A_r(u)\tp A_s(v))= {(r,s)}$ and
$\omega:V\to V$ by
$$\omega(i,j)=\sum_{a\in\Gamma_1}a(i,j).$$
\begin{prop}
\begin{enumerate}
\item The map $\f$ is an isomorphism of $\Lambda_0$-modules. \label{fiso}
\item The map $\omega$ is injective. \label{injective}
\item Let $r,s,t\in\N$. Then
\begin{align*}
  A_t\cdot(r,s)&= \omega^t(r,s), &
  B_t\cdot\left((r,s\!+\!1)+(r\!+\!1,s)\right)&=\omega^t(r\!+\!1,s\!+\!1) &\mbox{if}&\quad
  r\equiv s\equiv0, \\
  B_t\cdot(r,s)&= \omega^t(r,s), &
  A_t\cdot\left((r,s\!+\!1)+(r\!+\!1,s)\right)&=\omega^t(r\!+\!1,s\!+\!1) &\mbox{if}&\quad
  r\equiv s\equiv1.
\end{align*}
\label{omega} 
\end{enumerate}
\end{prop}

\begin{proof}
The proof of \ref{fiso} is an easy verification.

For \ref{injective}, consider a point $x=\sum_{i,j}\lambda_{i,j}{(i,j)}$ (where
$\l_{i,j}\in k$) in the kernel of $\omega$. 
Suppose that there exists a point $(l,m)\in \Gamma_0$ such that $\l_{l,m}\ne0$, and assume
that the natural number $m$ is minimal with this property.
We have either $\omega(l,m)=(l+1,m)+(l,m+1)$ or
$\omega(l,m)=(l+1,m)+(l,m+1)+(l+1,m+1)$, in both cases,
$\omega(\l_{l,m}{(l,m)})$ has a homogeneous component $\l_{l,m}(l+1,m)\ne0$.
Now $\omega(x-\l_{l,m}{(l,m)})=-\l_{l,m}\omega(l,m)\ne0$ hence, since
$\omega(i,j)\in\spann\{(i+1,j),(i,j+1),(i+1,j+1)\}$, it follows that either $\l_{l,m-1}$
or $\l_{l-1,m-1}$ is non-zero, contradicting the minimality of $m$.
Hence $x=0$, which proves that $\omega$ is injective.  

To prove \ref{omega}, we suppose that $r\equiv s$ and compute
\begin{align*}
  &\omega(r,s)= (r,s\!+\!1) + (r\!+\!1,s) + (r\!+\!1,s\!+\!1), \\
  &\omega\left((r,s\!+\!1) + (r\!+\!1,s)\right) = (r,s\!+\!2) + (r\!+\!2,s)
\end{align*}
hence
\begin{align*}
  A_1\cdot(r,s)&= Y\cdot(r,s)= \omega(r,s), &
  B_1\cdot\left((l,m\!+\!1)+(l\!+\!1,m)\right)&=\omega(l\!+\!1,m\!+\!1) &\mbox{if}&\quad
  r,s\equiv0, \\
  B_1\cdot(r,s)&= X\cdot(r,s)= \omega(r,s), &
  A_1\cdot\left((l,m\!+\!1)+(l\!+\!1,m)\right)&=\omega(l\!+\!1,m\!+\!1) &\mbox{if}&\quad
  r,s\equiv1.
\end{align*}
Now let $r,s\equiv0$, $u\ge1$ and assume, by induction, that \ref{omega} holds for all 
$t<u$. Then
\begin{align*}
B_u(r,s)&= A_{u-1}Y(r,s)=A_{u-1}\left( (r,s\!+\!1)+(r\!+\!1,s)+(r\!+\!1,s\!+\!1) \right)\\
 &= \omega^{u-1}\left( (r,s\!+\!1)+(r\!+\!1,s) \right) + \omega^{u-1}(r\!+\!1,s\!+\!1) =
 \omega^u(r,s) \,.
\end{align*}
The other identities are proved similarly.
\end{proof}

\begin{lma} \label{Qlma}
  $$\omega^t(0,0)=\sum_{l,m\in\N}Q_t^{(l,m)}(l,m)$$
  where
  \begin{align} \label{Q}
%    P_j^{(l,m)} &= \binom{l+m-2j}{l-j} \binom{\rned{\frac{l+m}{2}}}{j}, \\
    Q_t^{(l,m)} &= \binom{2t-l-m}{t-m} \binom{\rned{\frac{l+m}{2}}}{l+m-t} =
    \binom{2t-l-m}{t-l} \binom{\rned{\frac{l+m}{2}}}{l+m-t}
  \end{align}
is the number of paths in $\Gamma$ from $(0,0)$ to $(l,m)$ of length $t$. 
\end{lma}

\begin{proof}
Defining $Q_t^{(l,m)}$ as the number of paths of length $t$ from $(0,0)$ to $(l,m)$, the
expression for $\omega^t(0,0)$ certainly holds.

Let $P_j^{(l,m)}$ be the number of paths from $(0,0)$ to $(l,m)$ containing
precisely $j$ diagonal arrows $\gamma_{r,s}$. Then $Q_t^{(l,m)}=P_{l+m-t}^{(l,m)}$.

To find an expression for $P_j^{(l,m)}$, consider a path from $(0,0)$ to $(l,m)$ with $j$
arrows of type $\gamma$, it then has $l-j$ arrows of type $\a$ and $m-j$ arrows of type
$\b$. For every point $(r,s)$ in $\Gamma$, there is one arrow of type $\a$ and one arrow
of type $\b$ starting in $(r,s)$ (namely, $\a_{r,s}$ respectively $\b_{r,s}$), hence there
are $\binom{l+m-2j}{l-j}$ ways of choosing the mutual order of all arrows of type $\a$
respectively $\b$. 

Next choose where in the string of arrows $\a$ and $\b$ that the arrows $\gamma$ are to be 
inserted. A $\gamma$ could be inserted at points preceded by an even number of arrows $\a$
and $\b$, and multiple arrows $\gamma$ could be inserted at each point. 
There are $\rned{\frac{l+m}{2}}$ such points, so the number of choices is
$\binom{\rned{\frac{l+m}{2}}}{j}$.
This proves that $P_j^{(l,m)}=\binom{l+m-2j}{l-j} \binom{\rned{\frac{l+m}{2}}}{j}$, and
hence 
$Q_t^{(l,m)}=P_{l+m-t}^{(l,m)}=\binom{2t-l-m}{t-m} \binom{\rned{\frac{l+m}{2}}}{l+m-t}$.
\end{proof}

The following properties of $Q_t^{(l,m)}$ are easily derived from the definition:

\begin{align}
  \label{symm}
Q_t^{(l,m)} &= Q_t^{(m,l)}, \\
  \label{notime}
Q_t^{(l,m)} &\ne0 \quad\mbox{only if}\quad \max\{l,m\}\le t\le l+m, \\
  \label{diag}
Q_t^{(l,l)} &= \binom{2(t-l)}{t-l} \binom{l}{2l-t}=\d_{t,l}, \\
  \label{verge}
Q_t^{(l,0)} &= \binom{2t-l}{t} \binom{\rned{\frac{l}{2}}}{l-t} = \d_{t,l},\\
Q_{l+m}^{(l,m)} &= \binom{l+m}{l}. \label{j=0}
\end{align}

\section{Back diagonality} \label{bdlma}

The following proposition plays a key role in the determination of the Loewy length of
tensor products of uniserial modules. While Proposition~\ref{fiso} %and Lemma~\ref{Qlma}
tells us that the length of a module generated by a homogeneous basis element is
expressible in terms of the elements $\omega^t(0,0)=\sum_{l,m\in\N}Q_t^{(l,m)}(l,m)$,
$t\in\N$, the result below basically means that all terms in this sum except the ones for
which $l+m=t$ can be disregarded.

\begin{prop} \label{backdiag}
Let $t\in\N$ and suppose that $Q_t^{(l,m)}\equiv 1$. Then there exists $l'\le l$, and
$m'\le m$ such that $l'+m'=t$ and $Q_t^{(l',m')}\equiv1$.
\end{prop}

The remainder of this subsection is dedicated to the proof of Proposition~\ref{backdiag}.
The setup is the following: Assume that $Q_t^{(l,m)}\equiv1$ and $l+m-t=j>0$ (so that
$P_j^{(l,m)}\equiv1$). Then we need to show that there exist $l'\le l$ and $m'\le m$ such that
$Q_t^{(l',m')}\equiv1$ and $j'=l'+m'-t<j$, whence the statement of
Proposition~\ref{backdiag} follows by induction. Observe that if $l+m=t+j$ then
\begin{equation}
  Q_t^{(l,m)}=\binom{t-j}{l-j}\binom{\rned{\frac{t+j}{2}}}{j}\equiv
  \binom{t-j}{l-j}\binom{t+j}{2j} = \binom{t+j}{l+j}\binom{l+j}{2j} \,,
\end{equation}
hence $Q_t^{(l,m)}\equiv1$ if and only if $[t+j]_i\ge[l+j]_i\ge[2j]_i$ for all $i\in\N$.

Let $s=\max\{\sigma\in\N \mid [j]_\sigma=1\}$, \ie, $s$ is the highest number such that
$j\ge2^s$.

\begin{lma}
  If $2^s \nmid t$ then there exist $l'\le l$ and $m'\le m$ such that $l'+m'=t$ and
  $Q_t^{(l',m')}\equiv1$.
\end{lma}
\begin{proof}
Since $Q_t^{(l,m)}\equiv1$, we have $Q_{t-1}^{(l,m-1)}\equiv1$ or
$Q_{t-1}^{(l-1,m)}\equiv1$ or
$Q_{t-1}^{(l,m-2)}\equiv Q_{t-1}^{(l-1,m-1)}\equiv Q_{t-1}^{(l-2,m)}\equiv1$.
By induction, we may assume that there exists $\l_0\le l$ and $\mu_0\le m$ such that
$\l_0+\mu_0=t-1$ and $Q_{t-1}^{(\l_0,\mu_0)}\equiv1$. Applying $\omega$, we see that
either $Q_{t-1}^{(\l,\mu)}\equiv1$ 
for all $\l\le l$, $\mu\le m$ such that $\l+\mu=t-1$, or otherwise there exist
$l'\le l$, $m'\le m$ satisfying $l'+m'=t$ and $Q_t^{(l',m')}\equiv1$.

Suppose $Q_{t-1}^{(\l,\mu)}\equiv1$ for all $\l\le l$, $\mu\le m$ such that
$\l+\mu=t-1$. Observe that $Q_{t-1}^{(\l,\mu)}\equiv\binom{t-1}{\l}$ whenever
$\l+\mu=t-1$, so the above implies that
$\binom{t-1}{l-r}\equiv1$ for all $r\in\{0,\ldots,j+1\}$.
By Lucas' theorem, $\binom{t-1}{l-r}\equiv1$ is equivalent to $[t-1]_i\ge[l-r]_i$ 
for all $i\in\N$. 

Since $j\ge2^s$, we have $2^s\mid (l-r+1)$ for some $r\le j$, that is, $[l-r]_i=1$ for all
$i\in\{0,s-1\}$. 
Now $[t-1]_i\ge[l-r]_i$ implies that $[t]_i=[(t-1)+1]_i=0$  for all $i\le s-1$, \ie,
$2^s\mid t$. The result follows.
\end{proof}

From here on, we assume $2^s\mid t$.

\begin{lma} \label{jlemma}
  \begin{enumerate}
  \item $j=2^{s+1}-2^a$ for some $a\in\{0,\ldots,s\}$,
  \item $2^{s+1}\mid  t$ or $j=2^s$.
  \end{enumerate}
\end{lma}
\begin{proof}
1. The statement amounts to that $j=\sum_{i=a}^s2^i$. 
% Remember that $[j]_s=1$ by assumption.
Since $Q_t^{(l,m)}\equiv1$, we have $\binom{t+j}{2j}\equiv1$ and, by Lucas' theorem,
$$\binom{t+j}{2j}\equiv \prod_i\binom{[t+j]_i}{[2j]_i}=
\prod_{i=0}^{s-1}\binom{[j]_i}{[j]_{i-1}}\prod_{i\ge s}\binom{[t+j]_i}{[j]_{i-1}} \,.
$$
In particular, this implies $[j]_i\ge[j]_{i-1}$ for all $i\in\{1,\ldots,s-1\}$. 
In addition, since $[j]_s=1$ by assumption, we get $j=\sum_{i=a}^s2^i$ for some
$a\le s$.

2. If $j\ne2^s$ then $a<s$, implying that
$$1=[j]_{s-1}=[2j]_s\le[t+j]_s=1-[t]_s$$
that is, $[t]_s=0$. Hence $2^{s+1}\mid t$. 
\end{proof}

\begin{lma} 
  If $2^{s+1}\mid t$ then 
  \begin{align*}
    [t]_{s+1}=[l+j]_r&=1 &&\mbox{for all}\quad r\in\{a+1,\ldots,s+1\}, \quad\mbox{and} \\
     [l+j]_r&=0         &&\mbox{for all}\quad r\in\{0,\ldots,a-1\} \,.
\end{align*}
\end{lma}
\begin{proof}
If $2^{s+1}\mid t$ then the binary expansions of $j$ and $t$ are disjoint, so
$[t+j]_i=[j]_i$ if  $i\le s$ and $[t+j]_i=[t]_i $ if $i>s$.
The result then follows from the inequality $[t+j]_i\ge[l+j]_i\ge[2j]_i$, using
Lemma~\ref{jlemma}:1. 
\end{proof}

We now have all tools needed to prove Proposition~\ref{backdiag}. We will proceed in three
separate cases, namely:
\begin{enumerate}
\item $2^{s+1}\mid t$,\quad $2^{a+1}\mid l+j$,
\item $2^{s+1}\mid t$,\quad $2^{a+1}\nmid l+j$,
\item $2^{s+1}\nmid t$.
\end{enumerate}

\subsection{The case $2^{s+1}\mid t$ and $2^{a+1}\mid l+j$}

Set $l'=l-2^a$, $j'=j-2^a$ and $m'=t+j'-l'$. Now
\begin{align*} \binom{t+j'}{l'+j'}&=\binom{t+j-2^a}{l+j-2^{a+1}}\equiv1
  \,,&\mbox{since}&\quad  
\begin{cases} [t+j]_a=1,\;[l+j]_a=0, \\ [l+j]_{a+1}=1, \end{cases} \\
\intertext{and}
\binom{l'+j'}{2j'}&=\binom{l+j-2^{a+1}}{2j-2^{a+1}}\equiv1 \,,& \mbox{since}&\quad
[2j]_{a+1}=[l+j]_{a+1}=1,
\end{align*}
and, consequently, $Q_t^{(l',m')}\equiv1$.

\subsection{The case $2^{s+1}\mid t$, $2^{a+1}\nmid l+j$}

Here taking $l'=l$, $j'=j-2^a$, we get
\begin{align*} \binom{t+j'}{l'+j'}&=\binom{t+j-2^a}{l+j-2^a}\equiv1
  \,,&\mbox{since}&\quad  [t+j]_a=1=[l+j]_a, \\
\intertext{and}
\binom{l'+j'}{2j'}&=\binom{l+j-2^a}{2j-2^{a+1}}\equiv1 \,,& \mbox{since}&\quad
[2j]_{a+1}=1.
\end{align*}
Hence $Q_t^{(l',m')}\equiv1$ for $m'=t+j'-l'$.

\subsection{The case $2^{s+1}\nmid t$}

Since $2^{s+1}\nmid t$, we have $[t]_s=1$ and, by Lemma~\ref{jlemma}:2, $j=2^s$. 
This implies $[t+j]_r=[l+j]_r=0$ for all $r\le s$, and $[t+j]_{s+1}=[l+j]_{s+1}=1$.
Hence $[t]_{s+1}=[l]_{s+1}=0$, and $[t]_r=[t+j]_r$, $[l]_r=[l+j]_r$ for $r>s+1$. It
follows that $\binom{t}{l}\equiv1$, that is, $Q_t^{(l,t-l)}\equiv1$. 
\vspace{2ex}

This concludes the proof of Proposition~\ref{backdiag}.

\section{Solution to the uniserial case} \label{soluniser}

\begin{lma} \label{specials}
Let $l,m,s,t\in\N$.               
\begin{enumerate}
\item If $Q_{2^s}^{(l,m)}\equiv 1$ then $l=2^s$ or $m=2^s$. \label{2s}
\item $Q_{2^s-1}^{(l,m)}\equiv1$ for all $(l,m)$ such that $l+m=2^s-1$. 
  \label{2s-1}
\item If $l,m< 2^s\le t$ then $Q_t^{(l,m)}=0$. \label{ge2s}
\end{enumerate}
\end{lma}

\begin{proof}
The second statement follows directly form the formula
$Q_{l+m}^{(l,m)}\equiv\binom{l+m}{l}$, since $[2^s-1]_i=1$ for all $i\le s-1$.

If $l+m=2^s$ then $Q_{2^s}^{(l,m)}\equiv0$ unless $l\in\{0,2^s\}$, that is,
unless either $l$ or $m$ equals $2^s$.
Proposition~\ref{backdiag} implies that if $l,m\in\N$ are any numbers satisfying
$Q_{2^s}^{(l,m)}\equiv1$, then there exist $l'\le l$ and $m'\le m$ such that $l'+m'=2^s$ and
$Q_{2^s}^{(l',m')}\equiv1$.  So either $l'=2^s$ or $m'=2^s$, giving $l\ge2^s$ or $m\ge2^s$
respectively. 
On the other hand, by \ref{notime}, $Q_{2^s}^{(l,m)}=0$ if either of $l$ and $m$ is
greater than $2^s$, thus $2^s\in\{l,m\}$.

For the third statement, recall that $Q_t^{(l,m)}$ is the number of paths in
$\Gamma$ from $(0,0)$ to $(l,m)$ of length $t$. From this description follows that, for
every $r\le t$, 
$$Q_t^{(l,m)} = \sum_{\l,\mu}Q_r^{(\l,\mu)} \cdot \tilde{Q}_{t-r}^{(\l,\mu),(l,m)}$$ 
where $\tilde{Q}_{t-r}^{(\l,\mu),(l,m)}$ denotes the number of paths in $\Gamma$ from
$(\l,\mu)$ to $(l,m)$ of length $t-r$.
Clearly $\tilde{Q}_{t-r}^{(\l,\mu),(l,m)}=0$ if either $\l>l$ or $\mu>m$.
Setting $r=2^s$ we have, by the first statement, that $Q_{2^s}^{(\l,\mu)}\equiv0$ unless
$\l=2^s$ or $\mu=2^s$, but in this case $\tilde{Q}_{t-2^s}^{(\l,\mu),(l,m)}=0$.
Consequently, $Q_t^{(l,m)}\equiv0$.
\end{proof}

The key to determining the Loewy lengths in the uniserial case lies in the following
proposition, which provides an easy method to compute $\ell(u\tp v)$ for homogeneous basis
elements $u\tp v$.

\begin{prop} \label{hashprod}
For all $l,m\in\N$ the identity 
$\max\{t\in\N \mid \omega^t(0,0)\notin I(l,m)\} = l\#m$
holds.
\end{prop}

\begin{proof} 
Set $\tau(l,m)=\max\{t\in\N \mid \omega^t(0,0)\notin I(l,m)\}$.
First note that, by Lemma~\ref{specials}:\ref{ge2s}, if $l,m<2^a$ then
$\tau(l,m)<2^a$. Moreover, Lemma~\ref{specials}:\ref{2s-1} means that if $l+m\ge 2^a-1$ then
$\tau(l,m)\ge 2^a-1$. Hence, $\tau(l,m)=2^a-1=l\#m$ whenever $2^{a-1}\le l,m<2^a$.

Let $r,s,t\in\N$. From Proposition~\ref{backdiag} follows that $\tau(r,s)\ge t$ if, and only
if, there exist $\rho\le r$ and $\s\le s$ with $\rho+\s=t$ such that
$Q_t^{(\rho,\s)}\equiv1$.
Now suppose that $r,s\le 2^a$, and consider $\rho\le r$ and $\s\le s$. 
If $\rho+\s<2^a$ then 
$$Q_{\rho+\s}^{(\rho,\s)}\equiv\binom{\rho+\s}{\rho}=
\binom{2^a}{2^a}\binom{\rho+\s}{\rho}\equiv 
\binom{\rho+\s+2^a}{\rho+2^a}=Q_{\rho+2^a}^{(\rho+2^a,\s)}\,.$$
If instead $\rho+\s\ge 2^a$ then 
$Q_{\rho+\s}^{(\rho,\s)} \equiv 0 \equiv Q_{\rho+2^a+\s}^{(\rho+2^a,\s)}$ by 
Lemma~\ref{specials}:\ref{2s}.
In each case $Q_{\rho+\s}^{(\rho,\s)}\equiv Q_{\rho+2^a+\s}^{(\rho+2^a,\s)}$, which proves
that $\tau(r+2^a,s)=\tau(r,s) + 2^a$ for all $r,s<2^a$.
Taking $l=r+2^a$ and $m=s$ we get $\tau(l,m)=2^a + \tau(l-2^a,m)$ for all $l,m\in\N$ such that
$m<2^a\le l<2^{a+1}$.
On the other hand, from the definition follows that $l\#m=2^a+(l-2^a)\#m$ in this case;
now $\tau(l,m)=l\#m$ follows by induction.
\end{proof}

We record the following facts about the operation $\#$. The third statement is an
immediate consequence of Proposition~\ref{hashprod}, the others follow from the definition.
Remember that $\nu(x)=\min\{j\in\N \mid [x]_j\ne0\}$.

\begin{lma} \label{hashcor}
  Let $l,m\in\N$. 
  \begin{enumerate}
  \item $\max\{l,m\}\le l\#m\le l+m$. \label{hashineq}
  \item $l\#m= l+m \;\Leftrightarrow\; l\perp m$. \label{hashperp} 
  \item $\l\#\mu\le l\#m$ whenever $\l\le l$ and $\mu\le m$. \label{hashleq}
  \item If $l\nperp m$ then $l\#m=(l-1)\#m=l\#(m-1)$. \label{hashnotperp}
  \item If $l\perp m$ and $\nu(l)<\nu(m)$ then
    $l\#(m-1)<(l-1)\#m=l\#m-1=l+m-1$. \label{hashperp-1} 
  \end{enumerate}
\end{lma}

Formulae for the length of a tensor product of uniserial string modules can now be
derived from Proposition~\ref{hashprod}.

\begin{prop} \label{uniserprop}
  Let $l,m\in\N$, and let $s\in\N$ be the smallest number such that $[l]_r+[m]_r\le1$ for
  all $r\ge s$. Then
  \begin{align*}
    \ell(M(A_l)\tp M(B_m))&=
    \begin{cases}
      2+l\#m     &\mbox{if $l\not\perp m$}, \\
      1+l\#m = 1+l+m &\mbox{if $l\perp m$},
    \end{cases} \\
    \intertext{and}
    \ell(M(A_l)\tp M(A_m)) &=
    \begin{cases}
      2+l\#m &\mbox{if there exists $t<s-1$ such that $[l]_t=1$ or $[m]_t=1$}, \\
      1+l\#m &\mbox{if $[l]_t=[m]_t=0$ for all $t<s-1$},    
    \end{cases} \\
  \end{align*}
\end{prop}

Observe that $\ell(M(A_l)\tp M(A_m))\le \ell(M(A_l)\tp M(B_m))=\min\{2+l\#m,1+l+m\}$ for all
$l,m\in\N$.

\begin{proof}
Suppose that $M$ and $N$ are uniserial string modules, with standard bases $\B_M$ and
$\B_N$ respectively. Let $u\in\B_M$ and $v\in\B_N$ be top basis elements, and set
$l=h(u)$, $m=h(v)$. Denote by $u'$ the element in $\B_M$ satisfying
$h(u')=l-1$ (\ie, $u'=Xu$ or $u'=Yu$, whichever is non-zero), and by $v'$ the element in
$\B_N$ for which $h(v')=m-1$.
By Lemma~\ref{lem:Basisradical}, 
$h(M\tp N) = \max\{h(a\tp b) \mid a\in\B_M,b\in\B_N\}$.

Now, if $a\tp b\in\B_M\tp \B_N$ is pure (that is, if both $a$ and $b$ are annihilated by either $X$
or $Y$), then $h(a\tp b)=h(a)\#h(b)$. If on the other hand $Xa=0\ne Xb$, then 
$$h(a\tp b)=\max\{1+h(Ya\tp b),1+h(a\tp Xb)\}= 
\max\{1+(h(a)-1)\#h(b), 1+h(a)\#(h(b)-1)\}\,.$$

From Lemma~\ref{hashcor}:\ref{hashleq} follows that
$\max\{h(a\tp b)\mid a\in\B_M,b\in\B_N\}= \max\{h(u\tp v), h(u'\tp v), h(u\tp v')\}$.
Therefore, if $u\tp v$ is pure then 
\begin{align*}
h(M\tp N)&= \max\{l\#m, 1+(l-2)\#m, 1+(l-1)\#(m-1), 1+l\#(m-2) \}, \\
\intertext{whereas}
h(M\tp N) &= \max\{1+(l-1)\#m, 1+l\#(m-1), (l-1)\#m, l\#(m-1) \} \\
&= 1+\max\{(l-1)\#m,l\#(m-1) \}
\end{align*}
if $u\tp v$ is impure.

Starting with the impure case, assume that $M=M(A_l)$ and $N=M(B_m)$.
If $l\perp m$ then $\max\{(l-1)\#m, l\#(m-1)\}=l+m-1$, by
Lemma~\ref{hashcor}:\ref{hashperp} and \ref{hashperp-1}, hence $h(M\tp N)= l+m=l\#m$. 
If $l\not\perp m$ then $(l-1)\#m=l\#(m-1)=l\#m$ (Lemma~\ref{hashcor}:\ref{hashnotperp}),
so $h(M\tp N) = 1+l\#m$. 

It remains to consider the pure case. Let $M=M(A_l)$, $N=M(A_m)$, and let $s\in\N$ be the
smallest number such that $[l]_r+[m]_r\le1$ for all $r\ge s$, as stated in the proposition.
If $s=0$ then $l\#m=l+m$ whereas, by Lemma~\ref{hashcor}:\ref{hashineq},
$(l-2)\#m,(l-1)\#(m-1), l\#(m-2)\le l+m-2$. Hence $h(M\tp N)=l\#m$ in this case.

Assume instead that $s>0$, \ie, $l\not\perp m$. 
Then $[l]_{s-1}=[m]_{s-1}=1$, so we can write $l=\l+2^{s-1}+l_0$ and $m=\mu+2^{s-1}+m_0$
with 
\begin{equation*}
  \l=\sum_{t\ge s}[l]_t2^t,\quad l_0=\sum_{r<s-1}[l]_r2^r,\qquad\mbox{and}\qquad
  \mu=\sum_{t\ge s}[m]_t2^t,\quad m_0=\sum_{r<s-1}[m]_r2^r \,.
\end{equation*}
First, observe that if $l_0\ne0$ then $(l-1)\not\perp m$, and hence, by
Lemma~\ref{hashcor}:\ref{hashnotperp}, $(l-2)\#m=(l-1)\#m=l\#m$. Since
$(l-1)\#(m-1),l\#(m-2)\le l\#m$ by Lemma~\ref{hashcor}:\ref{hashleq}, it follows that
$h(M\tp N)=1+l\#m$. The same is true if $m_0\ne0$.

If $l_0=m_0=0$ then $[l-1]$ and $[m]$ are disjoint, so
$\max\{1+(l-2)\#m, 1+(l-1)\#(m-1), 1+l\#(m-2) \} \le l+m-1=(l-1)\#m\le l\#m$, implying that
$h(M\tp N)=l\#m$.

Summarising, we have
$$h(M(A_l)\tp M(A_m))=\begin{cases}l\#m &\mbox{if $s=0$ or if $s>0$ and $l_0=m_0=0$}, \\
1+l\#m & \mbox{if $s>0$ and either $l_0$ or $m_0$ is non-zero}.
\end{cases}$$
\end{proof}

\section{The non-uniserial case} \label{non-uniserial}

In the previous section, we computed the Loewy length of tensor products of uniserial
modules. In this section, tensor products involving bands with simple top and simple socle
are treated. 
Proposition~\ref{withuniserial} gives the Loewy length of the tensor product of such a
band with a uniserial string module, while Propositions~\ref{differentlegs},
\ref{differentlegs2} and \ref{bandband} cover the case of a product of two band modules.

\begin{prop} \label{withuniserial}
Let $N=M(A_m)$ and $M=M(A_{l_1}B_{l_2}\inv,\rho)$. Now
$$ \ell(M\tp N)=
\begin{cases}
  2+(l_1-1)\#m & \mbox{if $\rho=1$, $l_1=l_2$ and $l_1\perp m$, $l_1\perp(m-1)$,} \\
  \ell\left(M\left(A_{l_1}B_{l_2}\inv \right)\tp N\right) & \mbox{otherwise.}
\end{cases}$$
\end{prop}

Remember that $l\#m=l+m$ if and only if the binary expansions of $l$ and $m$ are
disjoint. 

\begin{proof}
Let $u_a$ and $u_b$ be top basis elements in $M(A_{l_1})$ and $M(B_{l_2})$ respectively,
and set $u'=u_a+u_b$. Then $\<u'\>\simeq M(A_{l_1}B_{l_2}\inv)$ and $M\simeq\<u'\>/U$, where
$U=\<A_{l_1}u'+\rho B_{l_2}u'\>=\<A_{l_1}u_a+\rho B_{l_2}u_b\>$.
We view the two latter isomorphisms as identifications, and write
$M(A_{l_1}B_{l_2}\inv)=\<u'\>$ and $M=M(A_{l_1}B_{l_2}\inv)/U$.
Moreover, $v$ denotes a top basis element in $N$. 
Now $M\tp N\simeq\frac{M(A_{l_1}B_{l_2}\inv)\tp N}{U\tp N}$, and
$\ell(M\tp N)\le \ell(M(A_{l_1}B_{l_2}\inv)\tp N)$.

Assume that $\ell(M\tp N)<\ell(M(A_{l_1}B_{l_2}\inv)\tp N)$. 
Then there exists a natural number $t$ such that 
$\rad^t \left(M(A_{l_1}B_{l_2}\inv)\tp N\right)$ is non-zero but contained in $U\tp N$.
Now, for any $t\ge0$, 
$B_t\cdot\left(M(A_{l_1}B_{l_2}\inv)\tp N\right) \not\subset U\tp N$ 
unless $B_t\cdot\left(M(A_{l_1}B_{l_2}\inv)\tp N\right) =0$,
so the above condition implies 
$0\ne A_t\cdot\left(M(A_{l_1}B_{l_2}\inv)\tp N\right)\subset U\tp N$ for some $t$.
The latter can be true only if $A_t\cdot(u'\tp v)\in (U\tp N)\minus\{0\}$,
equivalently, there exists a non-zero $n\in N$ such that 
$A_t\cdot(u_a\tp v) = A_{l_1}u_a\tp n$ and $A_t\cdot(u_b\tp v) = \rho B_{l_2}u_b\tp n$. 

We have $A_t\cdot(u_a\tp v)=\sum_{i,j}\gamma_{i,j}A_i u_a\tp A_j v$, with
$\gamma_{i,j}\in\mathbb{F}_2=\{0,1\}$. Hence $A_t\cdot(u_a\tp v)=A_{l_1}u_a\tp n$ implies 
$n=\sum_j \gamma_{l_1,j}A_jv$.
On the other hand, $A_t\cdot(u_b\tp v)=\rho B_{l_2}u_b\tp n$ similarly gives
$n=\sum_j \rho\d_j A_jv$ for some $\d_j\in\{0,1\}$. 
Since $n\ne0$ this means that $\rho=1$.

Next, by Equation~(\ref{notime}) and Proposition~\ref{backdiag}, the identity
$A_t\cdot(u_a\tp v) = A_{l_1}u_a\tp n$ implies that $\min\{j\in\N \mid \d_j\ne0\}=t-l_1$,
and $A_t\cdot(u_b\tp v) = B_{l_2}u_b\tp n$ implies 
$\min\{j\in\N \mid \d_j\ne0\}=t-l_2$.
Thus $l_1=l_2$.

From here on, set $l=l_1=l_2$.
The existence of a non-zero $n\in N$ such that $A_t\cdot(u_a\tp v) = A_lu_a\tp n$ and
$A_t\cdot(u_b\tp v) = B_lu_b\tp n$ 
is equivalent to $A_{l+m}(u_a\tp v)$ and $A_{l+m}(u_b\tp v)$ being
non-zero.
Namely, as we saw above, $A_t\cdot(u_a\tp v) = A_lu_a\tp n$
for $n=\sum_{\mu=0}^m\gamma_\mu A_\mu v\ne0$ implies that $\min\{\mu\mid
\gamma_\mu\ne0\}=t-l$.
Since $A_lu_a\in\soc M(A_l)$, we have $R\cdot(A_lu_a\tp n)=A_lu_a\tp Rn$ for all
$R\in\rad(kD_{4q})$, and so
$A_{t+r}\cdot(u_a\tp v) = A_lu_a\tp\left(\sum_\mu \gamma_\mu A_{\mu+r} v\right)$ for all
$r\in\N$.
Setting $r=l+m-t$, this gives
$$A_{l+m}\cdot(u_a\tp v) = A_lu_a\tp\left(\gamma_{t-l} A_m v\right)=
A_lu_a\tp A_m v\ne0.$$
Similarly, one proves that $A_{l+m}\cdot(u_b\tp v)$ is non-zero if
$A_t\cdot(u_b\tp v)=B_lu_b\tp n$ for some $t>0$ and $n\in N\minus\{0\}$.

Now $A_{l+m}(u_a\tp v)\ne0$ precisely when $l\#m=l+m$, and $A_{l+m}(u_b\tp v)\ne0$
precisely when $l\#(m-1)=l+m-1$. So if either doesn't hold, then 
$\ell(M\tp N)=\ell(\<u'\>\tp N)=\ell(M'\tp N)$. %=\ell(M(B_l)\tp N)$.

From here on, assume that $A_{l+m}(u_a\tp v)$ and $A_{l+m}(u_b\tp v)$ are non-zero.
Since the binary expansions of $l$ and $m$ are disjoint, $\nu(l)\ne\nu(m)$, and
$l\perp(m-1)$ then implies $\nu(l)>\nu(m)$. Consequently, $(l-1)\not\perp m$, so
$(l-1)\#m<l+m-1$ by Lemma~\ref{hashcor}:\ref{hashperp}.

We have 
\begin{equation} \label{geqB}
h(u\tp v)\ge \max\{\tau\mid B_\tau\cdot(u\tp v)\ne0\} =
\max\{\tau\mid B_\tau\cdot(M(B_l)\tp N)\ne0\}= 1+(l-1)\#m
\end{equation}
and it is easy to check that $h(a\tp b)\le h(u\tp v)$ for all $a\in\mathcal{B}_M$,
$b\in\mathcal{B}_N$, hence $h(M\tp N)=h(u\tp v)$ by Lemma~\ref{lem:Basisradical}.
Set $t=2+(l-1)\#m$. To prove the proposition, it is now enough to show that  
$A_t\cdot(u\tp v)=0$, since this, together with (\ref{geqB}), implies 
$h(u\tp v)=1+(l-1)\#m$.

Observe that
\begin{align*}
  &\min\{\tau\in\N\mid A_\tau\cdot(u\tp v)\in(\soc M) \tp N \} \\
  =&\min\{\tau\in \N\mid A_\tau\cdot(M(A_{l-1})\tp N)=0 \,,\; 
  A_\tau\cdot(M(B_{l-1})\tp N)=0 \} \\ 
  =&\max\{1+(l-1)\#m, 2+(l-1)\#(m-1)\} \le t \,,
\end{align*}
hence $A_t\cdot(u\tp v)\in(\soc M)\tp N$. Equivalently, 
$A_t\cdot(u_a\tp v)\in(\soc M(A_l))\tp N$ and $A_t\cdot(u_b\tp v)\in(\soc M(B_l))\tp N$.

Writing 
\begin{align*}
A_\tau \cdot(u_a\tp v)&=\sum_{\l,\mu}\a_\tau^{\l,\mu}(A_\l u_a)\tp(A_\mu v) \in M(A_l)\tp N
\quad\mbox{and}\\
A_\tau \cdot(u_b\tp v)&=\sum_{\l,\mu}\b_\tau^{\l,\mu}(B_\l u_b)\tp(A_\mu v) \in M(B_l)\tp N,
\end{align*}
we have $\a_\tau^{\l,\mu} \equiv Q_\tau^{(\l,\mu)}$ and
$\b_\tau^{\l,\mu} \equiv Q_{\tau-1}^{(\l,\mu-1)}\equiv\a_{\tau-1}^{\l,\mu-1}$.
As for $\tau=t$, 
\begin{align*}
A_t\cdot(u_a\tp v) &= \sum_{j=0}^{l+m-t}\a_t^{l,t+j-l}(A_lu_a)\tp(A_{t+j-l}v) \quad\mbox{and}
\\
A_t\cdot(u_b\tp v) &= \sum_{j=0}^{l+m-t}\b_t^{l,t+j-l}(B_lu_b)\tp(A_{t+j-l} v) \,.
\end{align*}
Both $A_t(u_a\tp v)$ and $A_t(u_b\tp v)$ are non-zero, since 
$\min\{\tau\mid A_\tau(u_a\tp v)=0\}= 1+l\#m=1+l+m>l+m\ge 2+(l-1)\#m=t$, and similarly
$\min\{\tau\mid A_\tau(u_b\tp v)=0\}= 2+l\#(m-1)>t$.
Hence, Proposition~\ref{backdiag} (together with the fact that
$\a_t^{\l,\mu}=\b_t^{\l,\mu}=0$ whenever $\l\ne l$) gives that $\a_t^{l,t-l}=1$ and
$\b_t^{l,t-l}=1$, that is, $\binom{t}{l}\equiv\binom{t-1}{l}\equiv1$. In particular,
$\nu(t)<\nu(l)$.

To conclude the proof of the proposition, we shall show that
$\a_t^{l,t+j-l}=\b_t^{l,t+j-l}$ for all $j\le l+m-t$, which is equivalent to 
$A_t(u\tp v)=0$.
As we have seen, $\a_t^{l,t-l}=\b_t^{l,t-l}=1$, so consider $j>0$.

First assume that $\a_t^{l,t+j-l}=1$ and $\nu(j)>\nu(t)$. Since also $\nu(l)>\nu(t)$, it
follows that $\nu(l+j)\ge\min\{\nu(l),\nu(j)\}\ge\nu(t)+1=\nu(t+j)+1$. Hence
$2^{\nu(t)+1}\mid(l+j)$ and $t+j=2^{\nu(t)} + T$, where $2^{\nu(t)+1}\mid T$.
So $Q_t^{(l,t+j-l)}\equiv \a_t^{l,t+j-l}=1$ implies
$\binom{T}{l+j}=\binom{T}{l+j}\binom{\nu(t)}{0}\equiv \binom{t+j}{l+j}\equiv1$.
%$$1\equiv\binom{t+j}{l+j}\equiv\binom{2^{\nu(t)}}{0}\prod_{i>\nu(t)}\binom{[t+j]}{}$$
But $t-1+j=T+(2^{\nu(t)}-1)$, so 
%2^{\nu(t)}=\sum_{i=0}^{\nu(t)-1}2^i
$$\binom{t-1+j}{l+j}\equiv\binom{T}{l+j}\binom{2^{\nu(t)}-1}{0}
%\prod_{i=0}^{\nu(t)-1}\binom{1}{0}
=\binom{T}{l+j}\equiv1 \,,$$
implying $\b_t^{l,t+j-l}\equiv\binom{t-1+j}{l+j}\binom{l+j}{2j}\equiv1$ 
(clearly, $1=\a_t^{l,t+j-l}\equiv\binom{t+j}{l+j}\binom{l+j}{2j}$ means that
$\binom{l+j}{2j}\equiv1$)

Next, consider the case $\nu(j)\le\nu(t)$.
Here, if $\binom{l+j}{2j}\equiv1$ then $[j]_i=1$ for all $i\in\{\nu(j),\ldots,\nu(l)-1\}$,
in particular, $[j]_{\nu(t)}=1$. Hence $[t+j]_{\nu(t)}=0$. But
$[l+j]_{\nu(t)}=[j]_{\nu(t)}=1$ since $\nu(l)>\nu(t)$, so $\binom{t+j}{l+j}\equiv0$ and
$\a_t^{l,t+j-l}=0$. 

The arguments in the two preceding paragraphs show that the inequality
$\a_t^{l,t+j-l} \le \b_t^{l,t+j-l}$ holds true.
To prove the converse assume, for a contradiction, that $\b_t^{l,\mu}=1$ and
$\a_t^{l,\mu}=0$, where $\mu=t+j-l$. As $\a_{t-1}^{l,\mu-1}=\b_t^{l,\mu}=1$, it follows
that either $\a_{t-1}^{l-1,\mu-1}=1$ or $\a_{t-1}^{l-1,\mu}=1$ (but not both).
By Proposition~\ref{backdiag}, there exist $\l_0\le l-1$ and $\mu_0\le \mu$ satisfying
$\l_0+\mu_0=t-1$ such that $\a_{t-1}^{\l_0,\mu_0}=1$. 
As $\a_t^{r,s}=0$ whenever $r\le l-1$, this means that actually $\a_{t-1}^{r,s}=1$ for
all $r\le l-1$, $s\le m$ satisfying $r+s=t-1$.
In particular, $\a_{t-1}^{l-1,t-l}=1$. But $\a_{t-1}^{l-1,t-l}=\b_t^{l-1,t-l+1}$,
and since $t-l+1=m-j+1\le m$, this contradicts the assumption that $\b_t^{l-1,s}=0$ for
all $s\le m$. 
This establishes the inequality $\b_t^{l,t+j-l}=1 \le \a_t^{l,t+j-l}$, concluding the
proof of the proposition.
\end{proof}

The last step in establishing the proof of Theorem~\ref{mainthm} is to determine the Loewy
length of a tensor product of two bands with simple top and socle.

First, we set some notation for the remainder of this section:
take $u_a$ and $u_b$ to be top basis elements in the modules $M(A_{l_1})$ and $M(B_{l_2})$
respectively, set $u'=u_a+u_b$ and $M'=\<u'\>\subset M(A_{l_1})\oplus M(B_{l_2})$.
Then $M'\simeq M(A_{l_1}B_{l_2}\inv)$, and setting 
$M=M'/U$, where 
$U=\<A_{l_1}u'+\rho B_{l_2}u'\>=\<A_{l_1}u_a+\rho B_{l_2}u_b\>\subset M'$ and 
$\rho\in k\minus\{0\}$, we have $M\simeq M(A_{l_1}B_{l_2}\inv,\rho)$.
Similarly, $v_a$ and $v_b$ are top basis element of $M(A_{m_1})$ respectively
$M(B_{m_2})$, $v'=v_a+v_b$, $N'=\<v'\>\subset M(A_{m_1})\oplus M(B_{m_2})$, and 
$N=N'/V$ for $V=\<A_{m_1}v'+\s B_{m_2}v'\>\subset N'$, 
$\s\in k\minus\{0\}$.
In this way, $N'\simeq M(A_{m_1}B_{m_2}\inv)$ and $N\simeq M(A_{m_1}B_{m_2}\inv,\s)$.
We denote by $\B_{M'}$ and $\B_{N'}$ the standard bases of $M'$ and $N'$ containing $u'$
and $v'$ respectively, and by $\B_M$ and $\B_N$ the corresponding standard bases in $M$
and $N$.
The images of $u'$ and $v'$ under the quotient projections $M'\to M$ and
$N'\to N$ are denoted by $u$ and $v$ respectively.

\begin{prop} \label{differentlegs}
Assume that $l_1\ne l_2$ and $m_1\ne m_2$. Then $\ell(M\tp N)=\ell(M'\tp N')$.
\end{prop}

\begin{proof}
Let $\overline{M}=\rad M\subset M$, and $\underline{M}=M/\soc M$.
Set $l=\max\{l_1,l_2\}$ and $m=\max\{m_1,m_2\}$, and let $x\in\{u_a,u_b\}$,
$y\in\{v_a,v_b\}$ be the unique elements such that $h(x)=l$ and $h(y)=m$.
Clearly,
\begin{equation} \label{inequality}
\max\{h(\overline{M}\tp N),h(\underline{M}\tp N)\}\le h(M\tp N)\le h(M'\tp N').
\end{equation}

Observe that
\begin{align*}
\rad^{l+m}(M'\tp N')&= \spann\{A_{l+m}\cdot(u'\tp v'), B_{l+m}\cdot(u'\tp v')\} \\
&=\spann\{A_{l+m}\cdot(x\tp y), B_{l+m}\cdot(x\tp y)\} 
\end{align*}
which is contained in $U\tp N'+ M'\tp V$ only if it is zero.
This means that $h(M\tp N)=l+m$ if, and only if, $h(M'\tp N')=l+m$, and hence the
proposition holds true in case $h(M'\tp N')=l+m$.

Assume instead that $h(M'\tp N')\le l+m-1$. Then, in particular, $l\#m\le l+m-1$, that is,
$l\not\perp m$. From Proposition~\ref{uniserprop} follows that 
$h(M'\tp N')\le 1+l\#m \le \min\{1+l\#m, l+m-1\}$.

By Proposition~\ref{subquocor} and Proposition~\ref{withuniserial},
\begin{align*}
&\max\{h(\overline{M}\tp N),h(\underline{M}\tp N)\}=
\max\{M(A_{l_i-1})\tp N, M(B_{l_i-1})\tp N \mid i=1,2 \} \\
=&\max\{M(A_{l-1})\tp N, M(B_{l-1})\tp N \} = \min\{1+(l-1)\#m, l+m-1\}.
\end{align*}
Since $l\not\perp m$, Lemma~\ref{hashcor} gives $(l-1)\#m=l\#m$, hence
$$\max\{h(\overline{M}\tp N),h(\underline{M}\tp N)\}= \min\{1+l\#m, l+m-1\} \ge 
h(M'\tp N').$$
This, together with the inequality (\ref{inequality}) shows that 
$h(M\tp N)=h(M'\tp N)=h(M'\tp N')$ in case $h(M'\tp N')\le l+m-1$, concluding the proof.
\end{proof}

\begin{prop} \label{differentlegs2}
  If $l_1\ne l_2$ and $m_1=m_2$, then $\ell(M\tp N)=\ell(M'\tp N)$.
\end{prop}

\begin{proof}
We have $\ell(M\tp N)<\ell(M'\tp N)$ if, and only if, there exists a number $t\in\N$ such
that $\rad^t(kD_{4q})\cdot(u'\tp v)$ is non-zero but contained in $U\tp N$. This is
equivalent to $A_t(u'\tp v)=(A_{l_1}u'+\rho B_{l_2}u')\tp n_1$ and 
$B_t(u'\tp v)=(A_{l_1}u'+\rho B_{l_2}u')\tp n_2$ with $n_1,n_2\in N$ not both equal to
zero.

Assume that $\ell(M\tp N)<\ell(M'\tp N)$, and let $n\in\{n_1,n_2\}$ be non-zero.
Set $\mu$ to be the largest number such that $n\in\rad^{\mu}(N)$.
Then $n=\sum_{j=\mu}^{m_1}\gamma_jA_jv + \sum_{j=\mu}^{m_1}\gamma'_jB_jv$, with either
 $\gamma_\mu$ or $\gamma'_\mu$ non-zero.
Now Proposition~\ref{backdiag} implies that $l_1+\mu=t=l_2+\mu$, hence $l_1=l_2$.
\end{proof}

From here on, assume that $l=l_1=l_2$ and $m=m_1=m_2$. 

\begin{prop} \label{bandband}\;\;\hfill
\begin{enumerate}
\item If $l\nperp m$, $l\nperp(m-1)$, $(l-1)\nperp m$ then $\ell(M\tp N)=
  2+(l-1)\#(m-1)=2+l\#m$.
\item If $l\perp m$, $(l-1)\perp m$, then
  $\ell(M\tp N)=
  \begin{cases}
    2+(l-1)\#(m-1) &\mbox{if}\quad \s=1,\\
    l+m+1 &\mbox{otherwise.}
  \end{cases}$
\item If $l\perp m$, $l\perp(m-1)$, then
  $\ell(M\tp N)=
  \begin{cases}
    2+(l-1)\#(m-1) &\mbox{if}\quad \rho=1,\\
    l+m+1 &\mbox{otherwise.}
  \end{cases}$
\item If $(l-1)\perp m$, $l\perp(m-1)$, then
  $\ell(M\tp N)=
  \begin{cases}
    2+(l-1)\#(m-1) &\mbox{if}\quad \rho=\s=1,\\
    l+m            &\mbox{if}\quad \rho=\s\ne1,\\
    l+m+1 &\mbox{otherwise.}
  \end{cases}$
\end{enumerate}
\end{prop}

\begin{proof}
We shall have reason to use the following formulae:
\begin{align}
A_{l+m}\cdot(u'\tp v') &= \d_1 A_lu'\tp A_mv' + 
\d_2 A_lu'\tp B_mv' + \d_3  B_lu'\tp A_mv', \label{Alm} \\
% \intertext{and}
B_{l+m}\cdot(u'\tp v') &= \d_1  B_lu'\tp B_mv' + 
\d_2  B_lu'\tp A_mv' + \d_3 A_lu'\tp B_mv', \label{Blm}
\end{align}
where $\d_1=\d_{l\#m,l+m}$, $\d_2=\d_{(l-1)\#m,l+m-1}$ and $\d_3=\d_{l\#(m-1),l+m-1}$. 
From this follows that
\begin{align}
A_{l+m}\cdot(u\tp v) &= (\sigma\rho\d_1 + \rho\d_2 + \sigma\d_3)  B_lu\tp B_mv, \label{Alm1} \\
% \intertext{and}
B_{l+m}\cdot(u\tp v) &= (\d_1   + 
\sigma\d_2 + \rho\d_3) B_lu\tp B_mv, \label{Blm2}
\end{align}
Set $\underline{M}=M/\soc M = M'/\soc M'$ and 
$\underline{N}=N/\soc N = N'/\soc N'$.
Observe that by Proposition~\ref{subquocor} and Proposition~\ref{withuniserial},
\begin{align*}
h(\underline{M}\tp N)&=\max\{h(M(A_{l-1})\tp N),h(M(B_{l-1})\tp N)\} = h(M(A_{l-1})\tp N) \\
&=
\begin{cases}
  1+(l-1)\#(m-1) &\mbox{if}\quad \s=1,\: (l-1)\perp m,\: (l-2)\perp m, \\
  h(M(A_{l-1})\tp M(B_m)) &\mbox{otherwise.}
\end{cases}
\end{align*}
Moreover, as $\underline{M}$ is a quotient of $M$, which in turn is a quotient of $M'$, we have the
following sequence of inequalities:
$$1+(l-1)\#(m-1)\le h(\underline{M}\tp N)\le h(M\tp N)\le h(M'\tp N)\le 
h(M'\tp N')=\min\{l+m,1+l\#m\},$$
where the last identity follows from Proposition~\ref{uniserprop}.
This immediately gives the first statement of the proposition:
if $l\nperp m$ and $(l-1)\nperp m$ (and hence $l\nperp(m-1)$) then $1+(l-1)\#(m-1)=1+l\#m$
by Lemma~\ref{hashcor}:\ref{hashnotperp}, so $h(M\tp N)=1+(l-1)\#(m-1)=1+l\#m$.

Assume now that $l\perp m$, $(l-1)\perp m$ (hence $l\nperp (m-1)$). If $\s=1$, then
$h(M'\tp N)=1+l\#(m-1)=1+(l-1)\#(m-1)$ where Proposition~\ref{withuniserial} gives the
first identity and Lemma~\ref{hashcor} the second.
It follows that $h(M\tp N)=1+(l-1)\#(m-1)$ in this case.
If on the other hand $\s\ne1$, then $A_{l+m}(u\tp v)=\rho(\sigma+1)B_lu\tp B_mv\neq 0$,
hence $h(M\tp N)=l+m$.
This proves the second statement of the proposition, whence the third statement follows by
symmetry.

For the fourth statement, assume that $(l-1)\perp m$ and $l\perp(m-1)$. 
Then, from the equations (\ref{Alm1}) and (\ref{Blm2}) follows\[A_{l+m}(u\tp v)=B_{l+m}(u\tp v)=(\rho+\sigma)B_lu\tp B_mv\] which is zero if, and only if, $\rho=\sigma$.
It follows that $h(M\tp N)=l+m$ if $\rho\ne\s$.
Assume instead $\rho=\s$. 
Since $h(a\tp b)\le h(a)+h(b)<l+m$ for any $a\tp b\in\B_{M}\tp\B_{N}\minus\{u\tp v\}$,
we have $h(M\tp N) < l+m$ in this case.
On the other hand, if $\rho\ne1$ then, by Proposition~\ref{withuniserial} and
Proposition~\ref{uniserprop}, 
$$h(M\tp N)\ge h(M\tp \underline{N})=h(M\tp M(A_{m-1}))=h(M(B_l)\tp M(A_{m-1}))= l+m-1 \,.$$
Hence $h(M\tp N)=l+m-1$ if $\rho=\s\ne1$.

It remains to consider the case $\rho=\s=1$.
First, suppose that $(l-2)\nperp m$; then $l+m>h(M\tp N)\ge h(\underline{M}\tp N)=l+m-1$, so
$h(M\tp N)=l+m-1$. But $(l-2)\nperp m$ implies that $l=\l+1$ and $m=\mu +1$ where $\l$ and
$\mu$ are even, positive integers satisfying $\l\perp \mu$; this means that 
$1+(l-1)\#(m-1)=1+\l+\mu=l+m-1=h(M\tp N)$.

Next, let $l=1$. Then again $h(\underline{M}\tp N)=m=l+m-1=1+(l-1)\#(m-1)$, and so 
$h(M\tp N)=1+(l-1)\#(m-1)$.

Last, assume that $(l-2)\perp m$, $l\ge2$. Then $l=\l+2^a$, $m=\mu+2^a$,
where $a\ge1$, $\l\perp\mu$ and $2^{a+1}\mid\l,\mu$; hence also $l\perp(m-2)$
holds. Consequently, $h(\underline{M}\tp N)=h(M\tp\underline{N})=1+(l-1)\#(m-1)$.
Setting $t=2+(l-1)\#(m-1)= \l+\mu+2^a+1$, we want to show that 
$\rad^t(M'\tp N')\subset U\tp N' + M'\tp V$.  

From Proposition~\ref{subquocor} follows that $h((\rad{M})\tp N)=h(\underline{M}\otimes
N)=t-1$ and $h(M\tp (\rad{N}))=h(M\tp \underline{N})=t-1$.  Thus, $h(a\tp b)\leq t-1$ for
any $a\tp b\in\B_{M}\tp\B_{N}\minus\{u\tp v\}$.  
It remains only to prove that $A_t(u'\tp v')$ and $B_t(u'\tp v')$ lie in 
$U\tp N' + M'\tp V$. 

The inequality $t>h(\underline{M}\tp N)$ implies that 
$\rad^t(M'\tp N')\subset (\soc M')\tp N' + M'\tp V$, and $t>h(M\tp\underline{N})$ gives
$\rad^t(M'\tp N')\subset U\tp N' + M'\tp(\soc N')$. Hence 
\begin{equation} \label{radinsoc}
\rad^t(M'\tp N')\subset U\tp N' + M'\tp V +(\soc M')\tp(\soc N').
\end{equation}

The restriction of the basis $\B_{M'}\tp\B_{N'}$ of $M'\tp N'$ to $\soc M'\tp\soc N'$ is
a basis of the latter, and the expansion of $A_t(u'\tp v')$ in $\B_{M'}\tp\B_{N'}$ has the
form 
$A_t\cdot(u'\tp v') = z+w$, where 
\begin{align*}
  z&=\a A_lu'\tp A_mv'+ \b A_lu'\tp B_mv'+\gamma B_lu'\tp A_mv' \in \soc M' \tp \soc N' \\
\intertext{with}
\a&=Q_t^{(l,m)},\quad \b=Q_{t-1}^{(l-1,m)},\quad
\gamma=Q_{t-1}^{
(l,m-1)}, \\
\intertext{and}
w&\in\spann\left((\B_{M'}\tp\B_{N'})\minus(\soc M'\tp\soc N')\right).
\end{align*}
By (\ref{radinsoc}), $w\in U\tp N' +M'\tp V$, so to show 
$A_t\cdot(u'\tp v')\in U\tp N' +M'\tp V$ it is enough to prove that 
$z\in U\tp N'+ M'\tp V$.

Remember that $l=\l+2^a$, $m=\mu+2^a$, with $\l\perp\mu$, \: $2^{a+1}\mid\l,\mu$ and
$a\ge1$. 
Setting $j=l+m-t$ we have
\begin{align}
\a&= Q_t^{(l,m)}\equiv\binom{t+j}{l+j}\binom{l+j}{2j}=
\binom{\l+\mu +2^{a+1}}{\l+2^{a+1}-1}\binom{\l+2^{a+1}-1}{2^{a+1}-2}\equiv 0, \label{alpha} \\ 
\intertext{since $\l+2^{a+1}-1\equiv1$, $\l+\mu+2^{a+1}\equiv0$, implying
$\binom{\l+\mu +2^{a+1}}{\l+2^{a+1}-1}\equiv0$,}
\b&= Q_{t-1}^{(l-1,m)}= \binom{t-1+j}{l-1+j}\binom{l-1+j}{2j}=
\binom{\l+\mu+2^{a+1}-1}{\l+2^{a+1}-2}\binom{\l+2^{a+1}-2}{2^{a+1}-2}\equiv 1, \label{beta}\\
\gamma&= Q_{t-1}^{(l,m-1)}=\binom{t-1+j}{l+j}\binom{l+j}{2j}=
\binom{\l+\mu+2^{a+1}-1}{\l+2^{a+1}-1}\binom{\l+2^{a+1}-1}{2^{a+1}-2}\equiv 1, \label{gamma}
\end{align}
where Lucas' theorem and the observation that 
$\binom{2^{a+1}-1}{2^{a+1}-2}=\binom{\sum_{i=1}^a2^i}{\sum_{i=2}^a2^i}\equiv1$ are used in
the two latter calculations.
This means that $z=A_lu'\tp B_mv'+B_lu'\tp A_mv' \in U\tp N' + M'\tp V$, proving that
$A_t\cdot(u'\tp v')\in U\tp N' + M'\tp V$ and hence $A_t\cdot(u\tp v)=0$.

Finally, 
$$B_t\cdot(u'\tp v')=\a B_lu'\tp B_mv'+ \b B_lu'\tp A_mv'+\gamma A_lu'\tp B_mv' +w'$$
where $w'\in \spann\left((\B_{M'}\tp\B_{N'})\minus(\soc M'\tp\soc N')\right)$. Again, Equation~(\ref{radinsoc}) gives 
$w'\in U\tp N' +M'\tp V$, and with (\ref{alpha})--(\ref{gamma}) we conclude that
$$B_t\cdot(u'\tp v')= B_lu'\tp A_mv'+ A_lu'\tp B_mv' +w' \in U\tp N' +M'\tp V.$$
 This proves that $\rad^t(M'\tp N')\subset U\tp N'+ M'\tp V$ and hence 
$h(M\tp N)=t-1=1+(l-1)\#(m-1)$.
\end{proof}

\section*{Acknowledgements}

The work for this article was partly carried out at the Mathematical Institute,
University of Oxford. The authors wish to thank Karin Erdmann for her encouragement and
support.

 \bibliographystyle{plain}
 
 %\bibliography{dihedralrefs}
\end{document}